\documentclass[a4, 11pt, twoside]{article}

\usepackage{amsmath, amscd, amsfonts, amssymb, amsthm, latexsym, url, color, todonotes, bm, framed, rotating, enumerate} 
\usepackage{graphicx}
\usepackage[left=1in, right=1in, top=1.2in, bottom=1in, includefoot, headheight=13.6pt]{geometry}
\usepackage{mathtools}
\input{xy}
\xyoption{all}

\usepackage[colorlinks,breaklinks=true]{hyperref}
\usepackage[figure,table]{hypcap}
\hypersetup{
	bookmarksnumbered,
	pdfstartview={FitH},
	citecolor={black},
	linkcolor={black},
	urlcolor={black},
	pdfpagemode={UseOutlines}
}
\makeatletter
\newcommand\org@hypertarget{}
\let\org@hypertarget\hypertarget
\renewcommand\hypertarget[2]{%
  \Hy@raisedlink{\org@hypertarget{#1}{}}#2%
} 
\makeatother 

\newtheorem{theorem}{Theorem}[section]
\newtheorem{lemma}[theorem]{Lemma}
\newtheorem{corollary}[theorem]{Corollary}
\newtheorem{proposition}[theorem]{Proposition}

\theoremstyle{definition}
\newtheorem{definition}[theorem]{Definition}
\newtheorem{remark}[theorem]{Remark}

\newcommand{\xysquare}[8]{
\[\xymatrix{
#1 \ar@{#5}[r] \ar@{#6}[d] & #2 \ar@{#7}[d]\\
#3 \ar@{#8}[r] & #4
}\]
}

\DeclareMathOperator*{\limone}{{\varprojlim}^1}
\DeclareMathOperator*{\indlimf}{``\varinjlim''}
\DeclareMathOperator*{\projlimf}{``\varprojlim''}

\newcommand{\al}{\alpha}
\newcommand{\bb}{\mathbb}

\newcommand{\blob}{\bullet}

\newcommand{\comment}[1]{}

\newcommand{\ep}{\varepsilon}

\newcommand{\into}{\hookrightarrow}

\newcommand{\isoto}{\stackrel{\simeq}{\to}}
\newcommand{\Isoto}{\stackrel{\simeq}{\longrightarrow}}

\newcommand{\onto}{\twoheadrightarrow}
\newcommand{\op}{\operatorname}
\newcommand{\pid}[1]{\langle #1 \rangle}
\renewcommand{\phi}{\varphi}
\newcommand{\quis}{\stackrel{\sim}{\to}}
\newcommand{\res}{\overline}
\newcommand{\roi}{\mathcal{O}}

\newcommand{\sub}[1]{{\mbox{\scriptsize #1}}}

\newcommand{\To}{\longrightarrow}
\newcommand{\ul}[1]{\underline{#1}}

\newcommand{\xto}{\xrightarrow}

\renewcommand{\cal}{\mathcal}
\renewcommand{\hat}{\widehat}
\renewcommand{\frak}{\mathfrak}
\newcommand{\indlim}{\varinjlim}
\renewcommand{\tilde}{\widetilde}
\renewcommand{\Im}{\operatorname{Im}}
\renewcommand{\ker}{\operatorname{Ker}}
\renewcommand{\projlim}{\varprojlim}

\DeclareMathOperator{\dlog}{dlog}

\DeclareMathOperator{\Spec}{Spec}
\DeclareMathOperator{\Spa}{Spa}

\newcommand{\xTo}[1]{\stackrel{#1}{\To}}

\usepackage{fancyhdr}

\usepackage{sectsty}
\sectionfont{\Large\sc\centering}
\chapterfont{\large\sc\centering}
\chaptertitlefont{\LARGE\centering}
\partfont{\centering}

\begin{document}
\itemsep0pt

\title{$p$-adic vanishing cycles as Frobenius-fixed points}

\author{Matthew Morrow}

\date{}

\maketitle

\begin{abstract}
Given a smooth formal scheme over the ring of integers of a mixed-characteristic perfectoid field, we study its $p$-adic vanishing cycles via de Rham--Witt and $q$-de Rham complexes, complementing some results of \cite{Bhatt_Morrow_Scholze3}.
\end{abstract}

\tableofcontents

\section{Introduction}
The main goal of this note is to study the relation of $p$-adic vanishing cycles to de Rham--Witt complexes, by interpreting the former as the Frobenius-fixed points of some of the integral cohomology theories from \cite{Bhatt_Morrow_Scholze2}. Our motivation for this is twofold. Firstly, we wished to understand the calculation of T.~Geisser and L.~Hesselholt \cite{GeisserHesselholt2006c} of $p$-adic vanishing cycles as the Frobenius-fixed points of a de Rham--Witt sheaf; their calculation works over a discretely valued $p$-adic field and takes into account the associated canonical log structure, whereas here we seek an analogue over the algebraic closure of the field. Secondly, we have tried to draw as close an analogy as possible with the situation of a smooth variety $Y$ over a perfect field $k$ of characteristic $p$, in which $p$-adic vanishing cycles should be replaced by the sheaf of de Rham--Witt log forms $W_N\Omega^j_{Y/k,\sub{log}}$ (i.e., the Frobenius-fixed points of the classical de Rham--Witt sheaves of $Y$), which is the same as the \'etale motivic cohomology $\bb Z(j)_\sub{\'et}/p^N\bb Z(j)_\sub{\'et}$ by T.~Geisser and M.~Levine \cite{GeisserLevine2000}. On the other hand, here we say nothing about the relation to syntomic cohomology or Nygaard filtrations; these are addressed in \cite{Bhatt_Morrow_Scholze3}.

To present the results we adopt the same set-up as \cite{Bhatt_Morrow_Scholze2}, namely we let $C$ be a perfectoid field of mixed characteristic containing all $p$-power roots of unity and $\frak X$ a smooth formal scheme over its ring of integers $\roi$; denote by $X$ its rigid analytic generic fibre over $C$. In \cite{Bhatt_Morrow_Scholze2} with B.~Bhatt and P.~Scholze we constructed certain complexes of sheaves $\tilde{W_r\Omega}_\frak X$ of $W_r(\roi)$-modules on the \'etale site $\frak X_\sub{\'et}$ (i.e., the \'etale site of its special fibre) for each $r\ge1$; a key property of these complexes is the existence of a certain ``$p$-adic Cartier isomorphism'' relating their cohomology sheaves to the relative de Rham complex $W_r\Omega^\blob_{\frak X/\roi}$ of A.~Langer and T.~Zink \cite{LangerZink2004}; to be precise, there are natural isomorphisms \[H^j(\tilde{W_r\Omega}_{\frak X}\{j\})\cong W_r\Omega^j_{\frak X/\roi}\] where $\{j\}$ denotes a ``Breuil--Kisin twist'' (this will be reviewed in Section \ref{subsection_BK}).

The complexes $\tilde{W_r\Omega}_{\frak X}\{j\}$ are moreover equipped with Frobenius and Restriction maps (the latter only after suitable cohomological truncation), lifting those on the de Rham--Witt sheaves $W_r\Omega^j_{\frak X/\roi}$, and the goal of this note is to study the Frobenius fixed points, to be precise the homotopy fibre of \[F-R:\tau^{\le j}\tilde{W_r\Omega}_\frak X\{j\}\To\tau^{\le j}\tilde{W_{r-1}\Omega}_\frak X\{j\}.\] The first theorem is the identification of this homotopy fibre modulo $p^N$, viewed as a pro complex of sheaves over the Restriction maps, as the $p$-adic vanishing cycles $\tau^{\le j}Rb_*(\bb Z/p^N\bb Z(j))$, where $b:X_\sub{\'et}\to\frak X_\sub{\'et}$ is the usual projection map of sites:

\begin{theorem}\label{theorem_intro}
For each $N,j\ge0$ there is a natural fibre sequence of pro complexes of sheaves on $\frak X_\sub{\'et}$ \[\tau^{\le j}Rb_*(\bb Z/p^N\bb Z(j))\To \projlimf_{r\sub{ wrt }R}\tau^{\le j}\tilde{W_r\Omega}_\frak X\{j\}\xTo{F-R}\projlimf_{r\sub{ wrt }R}\tau^{\le j}\tilde{W_{r-1}\Omega}_\frak X\{j\}.\]
\end{theorem}

We stress that when $\frak X$ arises as the $p$-adic completion of a smooth scheme over $\roi$ (or even via base change from a smooth scheme over the ring of integers of a discretely valued subfield of $C$), then $Rb_*(\bb Z/p^N\bb Z(j))$ is the usual complex of $p$-adic vanishing cycles $\res i^*R\res j_*(\bb Z/p^N\bb Z(j))$ by a comparison theorem of R.~Huber \cite[3.5.13]{Huber1996}; see Remark \ref{remark1} for the precise statement.

Suitably taking the inverse limit in Theorem \ref{theorem_intro} gives a proof of the following interpretation for $p$-adic vanishing cycles in terms of a twist $\bb A\Omega_\frak X\{j\}=\projlim_{r\sub{ wrt }F}\tilde{W_r\Omega}\{j\}$ of the main $q$-de Rham complex $\bb A\Omega_{\frak X}$ from \cite{Bhatt_Morrow_Scholze2}:

\begin{theorem}(Bhatt--M.--Scholze \cite{Bhatt_Morrow_Scholze3})\label{theorem_BMS}
For each $N,j\ge0$ there is a natural fibre sequence of pro complexes of sheaves on $\frak X_\sub{\'et}$ \[\tau^{\le j}R\nu_*(\bb Z/p^N\bb Z(j))\To \tau^{\le j}(\bb A\Omega_{\frak X}\{j\}/p^N)\xto{1-\phi^{-1}}\tau^{\le j}(\bb A\Omega_{\frak X}\{j\}/p^N).\]
\end{theorem}

Theorem \ref{theorem_BMS} was essentially known to the authors at the time of \cite{Bhatt_Morrow_Scholze2}, and we stress that passage through the complexes $\tilde{W_r\Omega}_\frak X\{j\}$ is an unnecessarily complicated proof of the result. Indeed, a more direct proof, as well as a strengthening in terms of the Nygaard filtration on $\bb A\Omega_\frak X\{j\}$, are presented in \cite[\S10]{Bhatt_Morrow_Scholze3}.

A necessary input for Theorem \ref{theorem_intro} is proving that the long exact sequence of $p$-adic vanishing cycles \begin{equation}\cdots \To R^ib_*(\bb Z/p\bb Z)\xTo{p^{N-1}} R^ib_*(\bb Z/p^N\bb Z)\To R^ib_*(\bb Z/p^{N-1}\bb Z)\To\cdots\label{eqn_BK}\end{equation} splits into short exact sequences. As far as we are aware, the only prior known proof of this result is to note that the necessary surjectivity is an immediate corollary of S.~Bloch and K.~Kato's classical result that $R^ib_*(\bb Z/p^N\bb Z)$ is generated by symbols for all $i,N\ge0$ \cite[Corol.~6.1.1]{Bloch1986}; see Remark \ref{remark_BK}. Another new proof, via $q$-de Rham complexes, is given in \cite[\S10]{Bhatt_Morrow_Scholze3}.

Passing to cohomology in Theorem \ref{theorem_intro}, and appealing to the $p$-adic Cartier isomorphism mentioned above, results in a long exact sequence of pro sheaves on $\frak X_\sub{\'et}$ terminating in \begin{equation}\cdots\To R^jb_*(\bb Z/p^N\bb Z(j))\xTo{\iota}\projlimf_{r\sub{ wrt R}}W_r\Omega^j_{\frak X/\roi}/p^N\xto{F-R}\projlimf_{r\sub{ wrt R}}W_{r-1}\Omega^j_{\frak X/\roi}/p^N\To 0\label{eqn_les}.\end{equation} Section \ref{section_Geisser_Hesselholt} explains how this relates to Geisser--Hesselholt's earlier result involving absolute log de Rham--Witt sheaves. The map $F-R:W_r\Omega^j_{\frak X/\roi}\to W_{r-1}\Omega^j_{\frak X/\roi}$ is the subject of the Appendix, where we prove it is surjective for a wide class of formal schemes by using infinitesimal deformation techniques to reduce to the case of a smooth variety over $\bb F_p$, in which case it is a classical theorem of L.~Illusie~\cite{Illusie1979}. 

We stress that the ``$\dlog$ map'' \[\iota:R^jb_*(\bb Z/p^N\bb Z(j))\To W_r\Omega^j_{\frak X/\roi}/p^N\] does not appear to be injective for any fixed $r\ge 1$, owing to $C$ being infinitely ramified. For a discussion of this issue and the unavoidable subtleties which it causes, we refer the reader to Section~\ref{subsection_iota}. To overcome this issue we consider the \'etale sheaf \[\overline{W\Omega}^j_{\frak X/\roi}:=H^j(\op{Rlim}_{r\sub{ wrt }R}\tau^{\le j}\tilde{W_r\Omega_{\frak X}}\{j\}),\] which is a certain extension of $W\Omega^j_{\frak X/\roi}$ by $\projlim^1_{r\sub{ wrt }R}W_r\Omega^{j-1}_{\frak X/\roi}\{1\}$ and which is equipped with a Frobenius endomorphism $F$ lifting the Frobenius on $W\Omega^j_{\frak X/\roi}$. Our final calculation of $p$-adic vanishing cycles as Frobenius-fixed points is as follows:

\begin{theorem}\label{theorem_GH}
For each $N,j\ge0$ there is a natural short exact sequence of \'etale sheaves on $\frak X$: \[0\To R^jb_*(\bb Z/p^N\bb Z(j))\To \overline{W\Omega}^j_{\frak X/\roi}/p^N\xto{1-F}\overline{W\Omega}^j_{\frak X/\roi}/p^N\To 0.\]
\end{theorem}

\subsubsection*{Layout of the paper}
Section \ref{section_prelim} is a brief recollection of the complexes of sheaves $\tilde{W_r\Omega}_\frak X$ from \cite{Bhatt_Morrow_Scholze2}, as well as Breuil--Kisin twists; the reader will likely have to consult \cite{Bhatt_Morrow_Scholze2} for further details.

The proofs of the main theorems are contained in Section \ref{section_main}. In Section \ref{subsection_Z} we introduce sheaves \[Z_r:=\ker(W_r(\hat\roi_X^+)\xto{R-F} W_{r-1}(\hat\roi_X^+))\] on the pro-\'etale site of $X$ and show that, for any $N\ge1$, the pro sheaf $\projlimf_rZ_r/p^NZ_r$ is precisely the constant sheaf $\bb Z/p^N\bb Z$. This allows us to calculate $p$-adic vanishing cycles as the Frobenius-fixed points of the pushforward of the pro sheaf $\projlimf_{r\sub{ wrt }R}W_r(\hat\roi_X^+)$. The next step is to show that this pushforward may be replaced by $\projlimf_{r\sub{ wrt }R}\tilde{W_r\Omega}_\frak X$; this is proved in Section \ref{subsection_more_as} by analysing various Artin--Schreier maps on their difference.

This almost completes the proof of Theorem \ref{theorem_intro}, except for a seemingly technical issue that modding out by $p^N$ does not a priori commute with truncation $\tau^{\le j}$ because of possible existence of $p$-torsion. In Section \ref{section_end} we present various equivalent conditions for this obstruction to vanish, one of which is exactly (\ref{eqn_BK}) breaking into short exact sequence. We then verify one the conditions, thereby completing the proof of Theorem \ref{theorem_intro}.

The long exact sequence (\ref{eqn_les}) follows by passing to cohomology sheaves, and an analysis of the final five terms of this sequence leads to Theorem \ref{theorem_GH}; this is done in Section \ref{subsection_extension}. Disjointly, taking the limit over Theorem \ref{theorem_intro} with respect to the Frobenius (rather than Restriction) maps formally establishes Theorem \ref{theorem_BMS}; this is done in Section \ref{subsection_nygaard}.

Finally, Section \ref{addendum} is included to answer informally certain questions which arise from our results, concerning non-injectivity of the map $\iota$ and the probable relation to Geisser--Hesselholt's result.

\subsubsection*{Acknowledgements}
I thank Peter Scholze for numerous discussions about $p$-adic vanishing cycles and filtrations on $\bb A\Omega_\frak X$, which were invaluable when proving Theorem \ref{theorem_intro}.

\section{Preliminaries}\label{section_prelim}
Throughout the article we adopt the following set-up. Let $C$ be a perfectoid field of mixed characteristic containing all $p$-power roots of unity, and $\frak X$ a smooth formal scheme over its ring of integers $\roi$; denote by $X$ its rigid analytic generic fibre over $C$. Let $\nu:X_\sub{pro\'et}\to \frak X_\sub{\'et}$ be the projection map from the pro-\'etale site of $X$ to the \'etale site of $\frak X$ (i.e., the \'etale site of its special fibre), and let $\hat\roi_X^+$ be the completed integral structure sheaf on $X_\sub{pro\'et}$. Here we use the pro-\'etale site of a locally Noetherian adic space as defined in \cite{Scholze2013}, and refer the reader also to \cite[\S5.1]{Bhatt_Morrow_Scholze2} for a summary.

\subsection{Breuil--Kisin twists}\label{subsection_BK}
We begin with a short recollection of Breuil--Kisin twists and refer the reader to \cite[\S3\&\S4.3]{Bhatt_Morrow_Scholze2} for details.

Let $\bb A_\sub{inf}:=W(\roi^\flat)$ be the infinitesimal period ring, $\theta:\bb A_\sub{inf}\to\roi$ Fontaine's map, and $\theta_r:\bb A_\sub{inf}(\roi)\to W_r(\roi)$ the generalisations of Fontaine's map for $r\ge1$; set $\tilde\theta_r:\bb A_\sub{inf}\to W_r(\roi)$ for $r\ge1$, which are compatible with the Frobenius map $F:W_r(\roi)\to W_{r-1}(\roi)$ and induce an isomorphism $\bb A_\sub{inf}\isoto\projlim_{r\sub{ wrt }F}W_r(\roi)$. In particular, each map  $\theta_r$ is surjective; moreover, its kernel is generated by a certain non-zero-divisor $\tilde\xi_r$, and therefore \[W_r(\roi)\{j\}:=(\ker\tilde\theta_r)^j/(\ker\tilde\theta_r)^{j+1}=\tilde\xi_r^j\bb A_{\sub{inf}}/\tilde\xi_r^{j+1}\bb A_{\sub{inf}}\] is an invertible $W_r(\roi)$-module for each $j\ge1$. These are the the {\em Breuil--Kisin twists} of $W_r(\roi)$.

Let $\zeta_p,\zeta_{p^2},\dots\in C$ be a compatible sequence of $p$-power roots of unity, $\ep:=(1,\zeta_p,\zeta_{p^2},\dots)\in\roi^\flat$, and let $\mu:=[\ep]-1\in\bb A_\sub{inf}$ be the unique element satisfying $\tilde\theta_r(\xi)=[\zeta_{p^{r}}]-1$ for all $r\ge1$. Let $\xi:=1+[\ep^{1/p}]+\cdots+[\ep^{1/p}]^{p-1}\in\bb A_\sub{inf}$ be the unique element satisfying $\tilde\theta_r(\xi)=\tfrac{[\zeta_{p^{r}}]-1}{[\zeta_{p^{r+1}}]-1}$ for all $r\ge1$; then $\xi$ is a generator of the kernel of Fontaine's map $\theta:\bb A_\sub{inf}\to\roi$, and the most convenient choice of generator of $\ker\tilde\theta_r$ for our purposes is $\tilde\xi_r:=\phi(\xi)\cdots\phi^r(\xi)$. We will also use $\xi_r:=\xi\phi^{-1}(\xi)\cdots \phi^{1-r}(\xi)$, which is a generator of $\ker\theta_r$.

The maps $\op{id},\phi^{-j}:\bb A_{\sub{inf}}\to\bb A_{\sub{inf}}$ induce maps $F_\sub{pre},R:W_r(\roi)\{j\}\to W_{r-1}(\roi)\{j\}$, which are semi-linear with respect to $F,R:W_r(\roi)\to W_{r-1}(\roi)$. Then $F_\sub{pre}$ has image $p^jW_{r-1}(\roi)\{1\}$, while $R$ has image $\tilde\theta_{r-1}(\xi)W_{r-1}(\roi)\{1\}$; set $F:=F_\sub{pre}/p:W_r(\roi)\{1\}\onto W_{r-1}(\roi)\{1\}$. Explicitly, if we trivialise $W_r(\roi)\{j\}$ via its basis element $\tilde\xi_r^j\op{mod} \tilde\xi_r^{j+1}\in W_r(\roi)\{j\}$, then the maps $F,R:W_r(\roi)\{j\}\to W_{r-1}(\roi)\{j\}$ identify respectively with $F,\,\tilde\theta_{r-1}(\xi^j)R:W_r(\roi)\to W_{r-1}(\roi)$; we will make frequent use of this trivialisation, while recalling that $\tilde\theta_{r-1}(\xi)=\tfrac{[\zeta_{p^{r-1}}]-1}{[\zeta_{p^{r}}]-1}$, to carry out calculations.

The submodule $([\zeta_{p^r}]-1)^jW_r(\roi)\{j\}$ of $W_r(\roi)\{j\}$ may be naturally identified with the Tate twist $W_r(\roi)(j)$, thereby providing a useful identification \[W_r(\roi)\{j\}=\tfrac{1}{([\zeta_{p^r}]-1)^j}W_r(\roi)(1)\] under which $F,R$ on the left correspond to $F,R$ on the right (obtained by localising $F,R$ on $W_r(\roi)$). This is proved by taking Tate modules of $\dlog:W_r(\roi)^\times\to\Omega^1_{W_r(\roi)/\bb Z_p}$.

More generally, suppose that for each $r\ge1$ we are given a $W_r(\roi)$-module $C_r$ (or complex, or sheaf, or complex of sheaves in a derived category, etc.) together with Frobenius and Restriction maps $F,R:C_r\to C_{r-1}$ which are compatible with $F,R:W_r(\roi)\to W_{r-1}(\roi)$. Then we may consider the twists $C_r\{j\}:=C_r\otimes_{W_r(\roi)}W_r(\roi)\{j\}$ and the twisted Frobenius and Restriction maps \[F:=F\otimes F,\quad R:=R\otimes R:C_r\{j\}\To C_{r-1}\{j\}.\] Trivialising as above, this explicitly means the maps $F, \tilde\theta_{r-1}(\xi)^jR:C_r\to C_{r-1}$. We will apply this in particular to the following possibilities of $C_r$: \[W_r(\hat\roi_X^+)\qquad R\nu_*W_r(\hat\roi_X^+)\qquad W_r\Omega^j_{\frak X/\roi}\]

The left-most object $W_r(\hat\roi_X^+)$ denotes the sheaf of rings on $X_\sub{pro\'et}$ obtained by taking Witt vectors of length $r$ of the completed integral structure sheaf $\hat\roi_X^+$, which is equipped with its usual Frobenius and Restriction map $F,R: W_r(\hat\roi_X^+)\to W_{r-1}(\hat\roi_X^+)$. In this case we can alternatively define the twist $W_r(\hat\roi_X^+)\{j\}=W_r(\hat\roi_X^+)\otimes_{W_r(\roi)}W_r(\roi)\{j\}$ as we did for $\roi$. Firstly the infinitesimal period sheaf is $\bb A_{\sub{inf},X}:=W(\hat\roi_{X^\flat}^+)$, where $\hat\roi_{X^\flat}^+$ is the tilted integral structure sheaf on $X_\sub{pro\'et}$; the surjections $\theta_r,\tilde\theta_r$ extend to $\bb A_{\sub{inf},X}\to W_r(\hat\roi_X^+)$ for each $r\ge1$, and their kernels are still defined by the non-zero divisors $\xi_r,\tilde\xi_r$ respectively; this shows that $W_r(\hat\roi_X^+)=(\ker\tilde\theta_r)^j/(\ker\tilde\theta_r)^{j+1}=\tilde\xi_r^j\bb A_{\sub{inf},X}/\tilde\xi_r^{j+1}\bb A_{\sub{inf},X}$.

\subsection{The sheaves $\tilde{W_r\Omega}_{\frak X}$}\label{section_bk2}
As in \cite[\S9]{Bhatt_Morrow_Scholze2} (to be precise, in \cite{Bhatt_Morrow_Scholze2} we worked throughout with the projection map $X_\sub{pro\'et}\to\frak X_\sub{Zar}$ rather than $X_\sub{pro\'et}\to\frak X_\sub{\'et}$, but the two resulting definitions of $\tilde{W_r\Omega}_{\frak X}$ coincide since the Zariski definition base changes correctly under any \'etale map; see \cite[Lem.~9.9 \& Corol.~9.11]{Bhatt_Morrow_Scholze2}, and also \cite[Eg.~4.21]{CesnaviciusKoshikawa2017}), we define \[\tilde{W_r\Omega}_{\frak X}:=L\eta_{[\zeta_{p^r}]-1}R\nu_*W_r(\hat\roi_X^+)\in D(W_r(\roi_\frak X)),\] where $L\eta$ is the d\'ecalage functor of \cite[\S6]{Bhatt_Morrow_Scholze2} with respect to the ideal generated by $[\zeta_{p^r}]-1\in W_r(\roi)$. Recall that there exists a canonical map $\tilde{W_r\Omega}_\frak X\to R\nu_*W_r(\hat\roi_X)$ \cite[Lem.~6.10]{Bhatt_Morrow_Scholze2}. We will systematically use the Breuil--Kisin twists \[\tilde{W_r\Omega}_{\frak X}\{j\}=\tilde{W_r\Omega}_{\frak X}\otimes_{W_r(\roi)}W_r(\roi)\{j\}=L\eta_{[\zeta_{p^r}]-1}R\nu_*W_r(\hat\roi_X)\{j\}\] of these complexes of sheaves in order to formulate certain statements in a natural (in particular, Galois equivariant when $\frak X$ is defined over a subfield of $C$) fashion. 

The Frobenius $F:W_r(\hat\roi_X^+)\{j\}\to W_{r-1}(\hat\roi_X^+)\{j\}$ formally induces a morphism after applying $R\nu_*$, and then ``restricts to'' the d\'ecalages since $F:W_r(\roi)\to W_{r-1}(\roi)$ sends $[\zeta_{p^r}]-1$ to $[\zeta_{p^{r-1}}]-1$:
\begin{equation}\xymatrix{
\tilde{W_r\Omega}_{\frak X}\{j\}\ar[r]^{\exists\,F}\ar[d]&\tilde{W_{r-1}\Omega}_{\frak X}\{j\}\ar[d]\\
R\nu_*W_r(\hat\roi_X^+)\{j\}\ar[r]_F& R\nu_*W_{r-1}(\hat\roi_X^+)\{j\}
}\label{squareF}\end{equation}
This may then be restricted to the cohomological truncations $F:\tau^{\le j}\tilde{W_r\Omega}_{\frak X}\{j\}\to\tau^{\le j}\tilde{W_{r-1}\Omega}_{\frak X}\{j\}$.

On the other hand, although the Restriction $R:W_r(\hat\roi_X^+)\{j\}\to W_{r-1}(\hat\roi_X^+)\{j\}$ similarly induces a morphism after applying $R\nu_*$ this does not restrict to $\tilde{W_r\Omega}_{\frak X}\{j\}\to \tilde{W_{r-1}\Omega}_{\frak X}\{j\}$, but only after taking the cohomological truncation:
\begin{equation}\xymatrix{
\tau^{\le j}\tilde{W_r\Omega}_{\frak X}\{j\}\ar[r]^{\exists\,R}\ar[d]&\tau^{\le j}\tilde{W_{r-1}\Omega}_{\frak X}\{j\}\ar[d]\\
R\nu_*W_r(\hat\roi_X^+)\{j\}\ar[r]_R& R\nu_*W_{r-1}(\hat\roi_X^+)\{j\}
}\label{squareR}\end{equation}
To prove this, note that we may factor the original $R$ as a \[W_r(\hat\roi_X^+)\{j\}\xto{\tfrac{1}{\tilde\theta_{r-1}(\xi)^j}R} W_{r-1}(\hat\roi_X^+)\{j\}\xto{\tilde\theta_{r-1}(\xi)^j} W_{r-1}(\hat\roi_X^+)\{j\},\] and then define the desired restriction map on the truncations to be the composition
\[\xymatrix{
\tau^{\le j}\tilde{W_r\Omega}_\frak X\{j\}\ar[d]\ar@{-->}[r]^R & \tau^{\le j}\tilde{W_{r-1}\Omega}_{\frak X}\{j\}\\
\tau^{\le j}R\nu_*W_r(\hat\roi_X^+)\{j\}\ar[r]_{\tfrac{1}{\tilde\theta_{r-1}(\xi)^j}R} & \ar[u]_{``\tilde\theta_{r-1}(\xi)^j"}\tau^{\le j}R\nu_*W_{r-1}(\hat\roi_X^+)\{j\}
}\]
where the right vertical arrow is the natural transformation of \cite[Lem.~6.9]{Bhatt_Morrow_Scholze2}.

Before continuing, we recall that the cohomology of $\tilde{W_r\Omega}_\frak X$, as well as the operators $R,F$ on its twists, are related to Langer--Zink's relative de Rham--Witt complex \cite{LangerZink2004} as follows:

\begin{theorem}[{\cite[\S11]{Bhatt_Morrow_Scholze2}}]\label{theorem_p-adic_Cartier}
There are natural isomorphisms of sheaves $H^j(\tilde{W_r\Omega}_\frak X\{j\})\cong W_r\Omega^j_{\frak X/\roi}$ on $\frak X_\sub{\'et}$ for all $j\ge0$ and $r\ge1$, such that $R,F$ on the left side (as constructed immediately above) correspond to the usual de Rham--Witt $R,F$ on the right side.
\end{theorem}

The theorem has a number of easy consequences which will be needed, so we explicitly include them here:

\begin{corollary}\label{corollary_torsion_bound}
\begin{enumerate}
\item The cohomology sheaves of $\tilde{W_r\Omega}_\frak X$ are $p$-torsion-free.
\item The canonical map $\tilde{W_r\Omega}_\frak X\to R\nu_*W_r(\hat\roi_X^+)$ induces an injection on each cohomology sheaf $H^i(-)$, with image $([\zeta_{p^r}]-1)^iR^i\nu_*W_r(\hat\roi_X^+)$.
\item The $[\zeta_{p^r}]-1$-power torsion (equivalently, the $p$-power-torsion) in $R^i\nu_*W_r(\hat\roi_X^+)$ is killed by $[\zeta_{p^r}]-1$ (hence by $p^r$).
\item For any $j\ge i\ge0$, the cokernel of $F-R:R^i\nu_*W_r(\hat\roi_X^+)\{j\}\to R^i\nu_*W_{r-1}(\hat\roi_X^+)\{j\}$ is killed by a power of $p$.
\end{enumerate}
\end{corollary}
\begin{proof}
First note that $p$-torsion-freeness is the same as $[\zeta_{p^r}]-1$-torsion-freeness over $W_r(\roi)$ by, e.g., \cite[Rmk.~3.16]{Morrow_BMSnotes}, which shows more precisely that $p^r\in ([\zeta_{p^r}]-1)W_r(\roi)$ and $([\zeta_{p^r}]-1)^{p^{r-1}}\in pW_r(\roi)$.

(i): Recall that $W_r\Omega^i_{\frak X/\roi}$ is $p$-torsion-free: indeed, since $\frak X$ is locally \'etale over a Laurent polynomial algebra and the relative de Rham--Witt complex behaves well under \'etale base change \cite[Lem.~10.8]{Bhatt_Morrow_Scholze2}, this reduces to $W_r\Omega^i_{\roi[T_1^{\pm1},\dots,T_d^{\pm1}]/\roi}$, which is $p$-torsion free thanks to the explicit description of the latter \cite[Thm.~10.12]{Bhatt_Morrow_Scholze2} and the fact that $W_r(\roi)$ is $p$-torsion free for all $r\ge1$.
 Since $H^i(W_r\Omega_{\frak X})$ is isomorphic to (a twist of) $W_r\Omega^i_{\frak X/\roi}$ by Theorem \ref{theorem_p-adic_Cartier}, this proves (i).

(ii): According to an elementary property of the d\'ecalage functor \cite[Lem.~6.9]{Bhatt_Morrow_Scholze2}, the canonical map \[H^i(\tilde{W_r\Omega}_{\frak X})=H^i(L\eta_{[\zeta_{p^r}]-1}R\nu_*W_r(\hat\roi_X^+))\To R^i\nu_*W_r(\hat\roi_X^+)\] has image equal to the multiples of $([\zeta_{p^r}]-1)^i$ and kernel killed by $([\zeta_{p^r}]-1)^j$. But we have just shown that $H^i(\tilde{W_r\Omega_{\frak X}})  $ is $[\zeta_{p^r}]-1$-torsion-free, so the canonical map is injective.

(iii): Another elementary property of the d\'ecalage functor \cite[Lem.~6.4]{Bhatt_Morrow_Scholze2} identifies $R^i\nu_*W_r(\hat\roi_X^+)$ modulo its $[\zeta_{p^r}]-1$-torsion with $H^i(W_r\Omega_{\frak X})$; since the latter has no $[\zeta_{p^r}]-1$-torsion, it follows that all $[\zeta_{p^r}]-1$-power torsion in the former is in fact killed by $[\zeta_{p^r}]-1$.

(iv): Trivialising the Breuil--Kisin twists in diagrams (\ref{squareF}) and (\ref{squareR}) via the basis element $\tilde\xi_r^j$ mod $\tilde\xi_r^{j+1}$, and taking cohomology $H^i(-)$ where $i\le j$, yields the following commutative diagrams:
\begin{equation}\xymatrix{
W_r\Omega_{\frak X/\roi}^i\ar[r]^{F}\ar@{^(->}[d]&W_{r-1}\Omega_{\frak X/\roi}^i\ar@{^(->}[d]\\
R^i\nu_*W_r(\hat\roi_X^+)\ar[r]_F& R^i\nu_*W_{r-1}(\hat\roi_X^+)
}\label{squareF'}\end{equation}

\begin{equation}\xymatrix{
W_r\Omega_{\frak X/\roi}^i\ar[r]^{\tilde\theta_{r-1}(\xi)^{j-i}R}\ar@{^(->}[d]&W_{r-1}\Omega_{\frak X/\roi}^i\ar@{^(->}[d]\\
R^i\nu_*W_r(\hat\roi_X^+)\ar[r]_{\tilde\theta_{r-1}(\xi)^{j}R}& R^i\nu_*W_{r-1}(\hat\roi_X^+)
}\label{squareR'}\end{equation}
The most relevant case of diagram (\ref{squareR'}) is when $i=j$ (this is also sufficient to establish it in general), in which case it exactly states that the isomorphisms of Theorem \ref{theorem_p-adic_Cartier} are compatible with the Restriction maps. Since $F-R:W_r\Omega_{\frak X/\roi}^i\to W_{r-1}\Omega_{\frak X/\roi}^i$ is surjective by Theorem \ref{theorem_surj_on_formal_schemes}, one see that the image of \[F-\tilde\theta_{r-1}(\xi)^{j-i}R=F-\left(\tfrac{[\zeta_{p^{r-1}}]-1}{[\zeta_{p^r}]-1}\right)^{j-i}R:W_r\Omega_{\frak X/\roi}^i\to W_{r-1}\Omega_{\frak X/\roi}^i\] contains all multiples of $([\zeta_{p^{r-1}}]-1)^{j-1}$. Appealing to the two diagrams and the image result in (ii), we see that the image of $F-\tilde\theta_{r-1}(\xi)^jR:R^i\nu_*W_r(\hat\roi_X^+)\to R^i\nu_*W_{r-1}(\hat\roi_X^+)$ contains all multiples of $([\zeta_{p^{r-1}}]-1)^j$, hence all multiples of $p^{rj}$ by the opening paragraph of the proof.
\end{proof}

\section{The main theorems}\label{section_main}
In this section we maintain the notation introduced at the start of Section \ref{section_prelim}; in particular, $\frak X$ is still a smooth formal scheme over the ring of integers $\roi$ of a perfectoid field of mixed characteristic containing all $p$-power roots of unity. Our first main goal is the following theorem, whose proof will occupy Sections \ref{subsection_Z}--\ref{section_end}:

\begin{theorem}\label{main_theorem_over_C}
For each $N,j\ge0$ there is a natural homotopy fibre sequence of pro sheaves on $\frak X_\sub{\'et}$ \begin{equation}\tau^{\le j}R\nu_*(\bb Z/p^N\bb Z(j))\To \projlimf_{r\sub{ wrt }R}\tau^{\le j}\tilde{W_r\Omega}_\frak X\{j\}/p^N\xTo{F-R}\projlimf_{r\sub{ wrt }R}\tau^{\le j}\tilde{W_{r-1}\Omega}_\frak X\{j\}/p^{N}.\label{eqn_better}\end{equation}
\end{theorem}

\begin{remark}\label{remark1}
\begin{enumerate}
\item Since the cohomology of $\tilde{W_r\Omega}_\frak X\{j\}$ is $p$-torsion-free by Corollary \ref{corollary_torsion_bound}(ii), the canonical map $(\tau^{\le j}\tilde{W_r\Omega}_\frak X\{j\})/p^N\to \tau^{\le j}(\tilde{W_r\Omega}_\frak X\{j\}/p^N)$ is a quasi-isomorphism for each $r,N\ge1$. Therefore we are justified in omitting parentheses in the above statement.

\item Suppose that $\frak X$ is the $p$-adic completion of a scheme $\cal X$. Then $R\nu_*(\bb Z/p^N\bb Z)$ (and similarly any Tate twist) may be identified with the usual $p$-adic vanishing cycles $\res i^*R\res j_*(\bb Z/p^N\bb Z)$, where \[\cal X_0:=(\cal X\otimes_{\bb Z}\bb Z/p\bb Z)_\sub{red}\xto{\res i}\cal X\stackrel{\res j}{\leftarrow}X_\eta:=\cal X\otimes_{\bb Z}\bb Z[\tfrac1p]\] denote the usual inclusions of the special and generic fibres.

To explain the proof it is convenient to write $\Lambda=\bb Z/p^N\bb Z$ with an appropriate subscript to denote the site on which this constant sheaf lies; essentially by definition we have in particular $\Lambda_{X_{\sub{pro\'et}}}=\nu^*\Lambda_{X_\sub{\'et}}$ and $\Lambda_{X_\sub{\'et}}=a^*\Lambda_{\cal X_{\eta\,\sub{\'et}}}$, where $a:X_\sub{\'et}\to \cal X_{\eta\,\sub{\'et}}$ is the morphism of sites constructed in \cite[3.5.12]{Huber1996}. Factoring $\nu:X_\sub{pro\'et}\to\frak X_\sub{\'et}$ as $X_\sub{pro\'et}\xto{w}X_\sub{\'et}\xto{b}\frak X_\sub{\'et}$ and using the quasi-isomorphism $\Lambda_{X_\sub{\'et}}\isoto R\nu_*\Lambda_{X_\sub{pro\'et}}$ of \cite[Corol.~3.17]{Scholze2013}, the problem is reduced to giving a natural quasi-isomorphism \[\res i^*R\res j_*\Lambda_{\cal X_{C\,\sub{\'et}}}\quis Rb_*a^*\Lambda_{\cal X_{\eta\,\sub{\'et}}}.\] But this is exactly Huber's comparison \cite[3.5.13]{Huber1996}.
\end{enumerate}
\end{remark}

Passing to cohomology sheaves in Theorem \ref{main_theorem_over_C}, which are described by Theorem \ref{theorem_p-adic_Cartier}, immediately yields the following (in which the surjectivity of the final map is a case of Theorem \ref{theorem_surj_on_formal_schemes}):

\begin{corollary}\label{corollary_les}
There is an associated long exact sequence of pro sheaves on the \'etale site of $\frak X$
\[\cdots\To R^{i}\nu_*(\bb Z/p^N\bb Z(j))\To \projlimf_{r\sub{ wrt }R}W_r\Omega^{i}_{\frak X/\roi}\{j-i\}/p^N\xto{F-R}\projlimf_{r\sub{ wrt }R}W_{r-1}\Omega^{i}_{\frak X/\roi}\{j-i\}/p^N\To\cdots\]
finishing in
\[\cdots\To R^j\nu_*(\bb Z/p^N\bb Z(j))\xTo{\iota}\projlimf_rW_r\Omega^j_{\frak X/\roi}/p^N\xto{F-R}\projlimf_rW_{r-1}\Omega^j_{\frak X/\roi}/p^N\To 0.\]
{\rm For discussion of the map $\iota$, see Section \ref{subsection_iota}.}
\end{corollary}

\subsection{The pro-\'etale sheaves $Z_r$}\label{subsection_Z}
We denote by $\bb Z/p^N\bb Z$ the pro-\'etale sheaf on $X_\sub{pro\'et}$ defined as $\nu^*(\bb Z/p^N\bb Z)$, where the latter $\bb Z/p^N\bb Z$ is the constant sheaf on $X_\sub{\'et}$; this minor abuse of notation should not cause confusion. Then, as in \cite[Def.~8.1]{Scholze2013}, we take the inverse limit on the pro-\'etale site to form the sheaf $\hat{\bb Z}_p:=\projlim_N\bb Z/p^N\bb Z$ on $X_\sub{pro\'et}$ (in fact, the sheaf $\hat{\bb Z}_p$ will not appear again after the statement of Lemma \ref{lemma_AS_maps}(ii)).

The following structure of the pro-\'etale site was implicit in \cite{Scholze2013}:

\begin{lemma}
A basis for the pro-\'etale site $X_\sub{pro\'et}$ is provided by the affinoid perfectoids $\cal U\in X_\sub{pro\'et}$ with the following additional property: any faithfully flat, finite \'etale $\hat \roi_X(\cal U)$-algebra admits a section.
\end{lemma}
\begin{proof}
A basis for the pro \'etale site is given by the connected affinoid perfectoids $\cal V$ \cite[Corol.~4.7]{Scholze2013} so it suffices, for each such $\cal V$, to find an affinoid perfectoid $\cal U$ with the desired property and a pro-\'etale cover $\cal U\to\cal V$. Without loss of generality we may choose a pro-\'etale presentation $\cal V=\projlimf_iV_i$ where
\begin{itemize}\itemsep0pt
\item each $V_i$ is affinoid, $V_i=\Spa(R_i,R_i^+)$ say;
\item all transition maps $V_{i'}\to V_i$ are finite \'etale;
\item the indexing set of the inverse system has a minimal element $0$;
\item $V_0$ is connected.
\end{itemize}
We now argue as in the proof of \cite[Thm.~4.9]{Scholze2013}. Let $\indlim_j S_j$ be a filtered colimit of connected $R_0$-algebras, along finite \'etale transition maps, such that any finite \'etale algebra over $\indlim_j S_j$ admits a section; let $S_j^+\subseteq S_j$ denote the integral closure of $R_0^+$ inside $S_j$. Then $\cal V':=\projlimf_j\Spa(S_j,S_j)^+\to\cal V$ is a pro-\'etale cover, and so $\cal U:=\cal V'\times_{V_0}\cal V\to\cal V$ is also a pro-\'etale cover by \cite[Lem.~3.10(i)]{Scholze2013}; we will show that $\cal U$ has the desired property.

Firstly, $\cal U$ is affinoid perfectoid by \cite[Lem.~4.5--4.6]{Scholze2013}. Next recall that base change induces an equivalence between finite \'etale algebras over $\roi_X(\cal U)=\roi_X^+(\cal U)=\indlim_{j,i} S_j\otimes_{R_0}R_i$ and over $\hat\roi_X(\cal U)=\hat\roi_X^+(\cal U)[\tfrac1p]$ \cite{GabberRamero2003}; then the usual $2\op{-lim}$ statement about finitely presented algebras (see e.g., \cite[Lem.~7.5]{Scholze2012}) tells us that any finite \'etale algebra over $\indlim_{j,i} S_j\otimes_{R_0}R_i$ comes via base change from a finite \'etale algebra over $(\indlim_j S_j)\otimes_{R_0}R_i$ for some fixed $i$ (we could even descend to a fixed $j$, but that would not be useful). These two descents respect faithful flatness (which, for finite \'etale algebras, is equivalent to injectivity) and therefore show that any faithfully flat finite \'etale $\hat\roi_X(\cal U)$ algebra comes via base change from a faithfully flat $(\indlim_j S_j)\otimes_{R_0}R_i$-algebra for some $i$. But $(\indlim_j S_j)\otimes_{R_0}R_i$ is finite \'etale over $\indlim_j S_j$, hence is isomorphic to a product of finitely many copies of $\indlim_j S_j$; therefore any faithfully flat finite \'etale algebra over it admits a section, which we may then base change back to the original $\hat\roi_X(\cal U)$-algebra.
\end{proof}

This structure of the pro-\'etale site easily yields the following results about various Artin--Schreier maps:

\begin{lemma}\label{lemma_AS_maps}
\begin{enumerate}
\item For each $N\ge1$, the sequence $0\to\bb Z/p^N\bb Z\to W_N(\hat\roi_{X^\flat}^+)\xto{\phi-1}W_N(\hat\roi_{X^\flat}^+)\to 0$ of sheaves on $X_\sub{pro\'et}$ is exact.
\item Taking $\projlim_N$, the resulting sequence $0\to \hat{\bb Z}_p\to \bb A_{\sub{inf},X}\xto{\phi-1} \bb A_{\sub{inf},X}\to 0$ on $X_\sub{pro\'et}$ is exact.
\item For $j\ge1$ and $\al\in\xi^j\bb A_\sub{inf}$, the operator $1-\al\phi^{-1}$ is an automorphism of the sheaf $\bb A_{\sub{inf},X}/\mu^j$.
\item The map $F-R:W_r(\hat\roi_X^+)\to W_{r-1}(\hat\roi_X^+)$ of sheaves on $X_\sub{pro\'et}$ is surjective.
\end{enumerate}
\end{lemma}
\begin{proof}
(i)\&(ii): Given an affinoid perfectoid $\cal U$ with the property of the previous lemma, writing $A=\roi_X(\cal U)$ and $A^+=\roi_X(\cal U)$, then the (easy direction of the) Almost Purity Theorem implies that any faithfully flat finite \'etale algebra over the perfectoid Tate ring $A^\flat$ admits a section. Since $A^{\flat+}$ is integrally closed in $A^\flat$, the Artin--Schreier maps on $W_N(A^{\flat+})=\Gamma(\cal U,W_N(\hat\roi^+_{X^\flat})$ and $W(A^{\flat+})=\Gamma(\cal U,\bb A_{\sub{inf},X})$ are therefore surjective. This gives (i) and (ii).

(iii): It is enough to prove that $1-\al\phi^{-1}$ is an automorphism of $W(A^{\flat+})/\mu$ for all $A^{\flat+}$ as in the first paragraph (since this proves that $1-\al\phi^{-1}$ is an automorphism on the sections of the presheaf quotient $\bb A_{\sub{inf},X}/\mu$ on a basis for the site, whence it is an automorphism after sheafifying). Write $\al=\xi^j\al'$ for some $\al'\in\bb A_\sub{inf}$, and consider the following diagram of short exact sequences on $X_\sub{pro\'et}$:
\[\xymatrix@C=2cm{
0\ar[d] & 0\ar[d]\\
W(A^{\flat+})\ar[r]^{\phi-\phi(\al')}\ar[d]_{\mu^j} & W(A^{\flat+})\ar[d]^{\phi(\mu)^j} \\
W(A^{\flat+})\ar[r]^{\phi-\phi(\al')\tilde\xi^j}\ar[d] & W(A^{\flat+})\ar[d]^{\phi^{-1}}\\
W(A^{\flat+})/\mu^j\ar[r]^{1-\al\phi^{-1}}\ar[d] & W(A^{\flat+})/\mu^j\ar[d]\\
0&0
}\]
We must prove that the top square is bicartesian; since all the terms in the diagram are all $p$-adically complete it is enough to prove this modulo $p$, i.e. that the square 
\[\xymatrix@C=2cm{
A^{\flat+}\ar[r]^{\phi-\phi(\al')}\ar[d]_{\mu^j} & A^{\flat+}\ar[d]^{\phi(\mu)^j} \\
A^{\flat+}\ar[r]^{\phi-\phi(\al')\tilde\xi^j} & A^{\flat+}
}\]
is bicartesian. But the horizontal arrows are surjective just as in the first paragraph (though note the degenerate case when $\al'$ vanishes mod $p$, in which case the horizontal arrows are no longer Artin--Schreier maps but rather the isomorphism $\phi$). Moreover, the integral closedness of $A^{\flat+}$ inside $A^\flat=A^{\flat+}[\tfrac1\mu]$ immediately implies the induced vertical arrow on the horizontal kernels is surjective. Since the vertical arrows are also injective the diagram is bicartesian, as desired.

(iv): This follows from the surjectivity in part (ii) and the following commutative diagram with surjective vertical arrows
\[\xymatrix{
\bb A_{\sub{inf},X}\ar[r]^{\phi-1}\ar@{->>}[d]_{\theta_r}&\bb A_{\sub{inf},X}\ar@{->>}[d]^{\theta_{r-1}}\\
W_r(\hat\roi_X^+)\ar[r]_{F-R}&W_{r-1}(\hat\roi_X^+)
}\]
(see \cite[Lem.~3.4]{Bhatt_Morrow_Scholze2}.)
\end{proof}

Let $Z_r$ denote the kernel of $F-R:W_r(\hat\roi_X^+)\onto W_{r-1}(\hat\roi_X^+)$, and note that the restriction $R: W_r(\hat\roi_X^+)\to W_{r-1}(\hat\roi_X^+)$ induces a map $Z_r\to Z_{r-1}$. For each $N\ge1$ there is a resulting commutative diagram of short exact sequences
\begin{equation}\xymatrix{
0\ar[r] & \bb Z/p^N\bb Z\ar[d]_{\vartheta_r^N}\ar[r] & W_N(\roi_X^{+\flat})\ar[r]^{\phi-1}\ar@{->>}[d]_{\theta_r}&W_N(\roi_X^{+\flat})\ar[r]\ar@{->>}[d]_{\theta_{r-1}}&0\\
0\ar[r] & Z_r/p^NZ_r\ar[r] & W_r(\hat\roi_X^+)/p^N\ar[r]_{F-R}&W_{r-1}(\hat\roi_X^+)/p^N\ar[r]&0
}\label{eqn_Z_bbZ}\end{equation}
where we have taken the final commutative diagram of the previous proof modulo $p^N$. By the first commutative diagram of \cite[Lem.~3.4]{Bhatt_Morrow_Scholze2}, the restriction maps $R$ on the bottom line are compatible with the identity maps on the top line, thereby giving rise to a map of pro sheaves on $X_\sub{pro\'et}$ \[\vartheta_\infty^N:\bb Z/p^N\bb Z\To\projlimf_{r} Z_r/p^NZ_r\] The following result states that this is an isomorphism, thereby allowing us to study $p$-adic \'etale cohomology via the sheaves $W_r(\hat\roi_X^+)$:

\begin{lemma}\label{lemma_p_adic_cyces_via_witt}
For any fixed $N\ge1$, the map of pro sheaves $\vartheta_\infty^N$ is an isomorphism. More precisely,
\begin{enumerate}
\item $\vartheta_r^N$ is injective; 
\item the image of $\vartheta_r^N$ contains (hence equals) the image of $R^N:Z_{r+N}/p^NZ_{r+N}\to Z_r/p^NZ_r$. 
\end{enumerate}
Hence, for any $r,N\ge1$, there exists a unique map $Z_{r+N}/p^NZ_{r+N}\to \bb Z/p^N\bb Z$ making the following diagram commute:
\[\xymatrix{
\bb Z/p^{N}\bb Z\ar[r]^{\op{id}}\ar[d]_{\vartheta_{r+N}^N} & \bb Z/p^N\bb Z\ar[d]^{\vartheta_r^N}\\
Z_{r+N}/p^NZ_{r+N}\ar[r]_{R^N}\ar[ur]^{\exists} & Z_r/p^NZ_r
}\]
\end{lemma}
\begin{proof}
(i): By composing with the map to $W_r(\hat\roi_X^+)/p^N$ and then the projection to $\hat\roi_X^+/p^N$, it is enough to show that $\bb Z/p^N\bb Z\to \hat\roi_X^+/p^N$ is injective. This is $\nu^*$ of the sheafification of the following map of presheaves on $X_\sub{\'et}$: \[\text{constant presheaf }\bb Z/p^N\bb Z\To\text{presheaf quotient }\roi_{X_\sub{\'et}}^+/p^N\] Since $\nu^*$ and sheafification are exact, it is enough to show that this map of presheaves is injective on a basis, i.e., that for each affinoid $\Spa(R,R^+)\in X_\sub{\'et}$ the natural map $\bb Z/p^N\bb Z\to R^+/p^NR^+$ is injective; but this follows from the fact that $R^+$ is contained inside the set of elements of $R$ which are power bounded for the $p$-adic topology.

(ii): Let $\xi_r:=\phi^{-r}(\tilde\xi_r)\in\bb A_\sub{inf}$, which we recall is a generator for the kernel of $\theta_r=\tilde\theta_r\phi^r:\bb A_{\sub{inf},X}\to W_r(\hat\roi_X^+)$. Then the induced map on the kernels of the middle and right vertical arrows of diagram (\ref{eqn_Z_bbZ}) is \[\phi-1:\xi_rW_N(\roi_X^{+\flat})\To \xi_{r-1}W_N(\roi_X^{+\flat})\] (with the pro sheaf structure corresponding to the natural inclusions $\cdots\subseteq\xi_{r+1}W_N(\roi_X^{+\flat})\subseteq \xi_rW_N(\roi_X^{+\flat})\subseteq\cdots$), so we will analyse the sections of this on any affinoid perfectoid $\cal U$; let $A^{\flat+}:=\hat\roi^+_{X^\flat}$ be the corresponding tilted integral perfectoid ring. First observe that $\phi$ is contracting on multiples of $\mu$ and so $\phi-1:\mu W_N(A^{+\flat})\to\mu W_N(A^{+\flat})$ is an isomorphism; more precisely, this isomorphism follows from the facts that $\phi^r(\mu)=\tilde\xi_r\mu$ for all $r\ge1$ and that $W_N(A^{+\flat})\isoto\projlim_rW_N(A^{+\flat})/\tilde\xi_r$. In the same way, the map $\phi^N-1:\mu W_N(A^{+\flat})\to\mu W_N(A^{+\flat})$ is also an isomorphism.

Next observe that $\phi^N(\xi_{r+N})$ is divisible by $\phi^N(\xi_N)$, which is $\equiv p^N=0$ mod $\mu$; that is, $\phi^N(\xi_{r+N})$ is divisible by $\mu$ in $W_N(A^\flat)$. So the following diagram of short exact sequences is well-defined and commutative:
\[\xymatrix{
0\ar[d] & 0\ar[d]\\
\mu W_N(A^{+\flat})\ar[r]^{\phi^N-1}\ar[d] & \mu W_N(A^{+\flat})\ar[d] \\
\xi_{r+N} W_N(A^{+\flat})\ar[r]^{\phi^N-1}\ar[d] & \xi_{r+N} W_N(A^{+\flat})\ar[d]\\
\xi_{r+N} W_N(A^{+\flat})/\mu W_N(A^\flat)\ar[r]^{\op{id}}\ar[d] & \xi_{r+N} W_N(A^{+\flat})/\mu W_N(A^{+\flat})\ar[d]\\
0&0
}\]
Since the top and bottom horizontal arrows are isomorphisms, we deduce that the middle arrow is also an isomorphism. By writing $\phi^N-1=(\phi-1)(\phi^{N-1}+\cdots+\phi+1)$, this isomorphism implies that \[\phi-1:\xi_{r+N} W_N(A^{+\flat})\To \xi_{r+N-1} W_N(A^{+\flat})\] is injective, and that the image of \[\phi-1:\xi_{r} W_N(A^{+\flat})\To \xi_{r-1} W_N(A^{+\flat})\] contains $\xi_{r+N-1}W_N(A^{+\flat})$. In particular, the previous line is an isomorphism when we pass to pro abelian groups over $r\ge1$.

An easy diagram chase now gives (ii), and the existence of the commutative diagram then follows from elementary algebra.
\end{proof}

\subsection{Further Artin--Schreier maps}\label{subsection_more_as}
In the following results we continue to study various Artin--Schreier maps on pro-\'etale sheaves and pro-\'etale cohomology. From now on we will always regard $W_r(\roi)$ as an $\bb A_\sub{inf}$-algebra through $\tilde\theta_r$. In particular, this means that if $\al\in\bb A_\sub{inf}$, then the notation $\al R:W_r(\roi)\to W_{r-1}(\roi)$ denotes the map $\tilde\theta_{r-1}(\al)R$ (which is not a map of $\bb A_\sub{inf}$-modules according to this convention, as the restriction $R$ does not commute with the maps $\tilde\theta_r$).

The following underlies all further results in this subsection:

\begin{lemma}\label{lemma_F-R_surj}
Fix $\al\in\xi\bb A_\sub{inf}$. For any $r>1$, the map \[F-\al R:W_r(\hat\roi_X^+)/[\zeta_{p^r}]-1\To W_{r-1}(\hat\roi_X^+)/[\zeta_{p^{r-1}}]-1\] of pro-\'etale sheaves is surjective and its kernel $=\ker F\subseteq\ker \al R$. Hence there exist (necessarily unique) morphisms of sheaves $W_r(\hat\roi_X^+)/[\zeta_{p^r}]-1\to W_r(\hat\roi_X^+)/[\zeta_{p^r}]-1$ for $r\ge1$ making the following diagrams commute:
\[\xymatrix{
W_r(\hat\roi_X^+)/[\zeta_{p^r}]-1\ar[r]^\exists\ar[d]_{F-\al R}\ar[dr]^{\al R} & W_r(\hat\roi_X^+)/[\zeta_{p^r}]-1\ar[d]^{F-\al R}\\
W_{r-1}(\hat\roi_X^+)/[\zeta_{p^{r-1}}]-1\ar[r]^{\exists} & W_{r-1}(\hat\roi_X^+)/[\zeta_{p^{r-1}}]-1
}\]
\end{lemma}
\begin{proof}
Since $\tilde\xi_r\equiv p^r$ mod $\mu$, the isomorphism $\tilde\theta_r:\bb A_{\sub{inf},X}/\tilde\xi_r\isoto W_r(\hat\roi_X^+)$ (which sends $\mu$ to $[\zeta_{p^r}]-1$) induces modulo $\mu$ an isomorphism $\tilde\theta_r:W_r(\roi^{+\flat})/\mu\isoto W_r(\hat\roi_X^+)/[\zeta_{p^r}]-1$. These isomorphisms fit into a diagram
\[\xymatrix@C=3cm{
W_r(\roi_X^{+\flat})/\mu \ar[r]^{\tilde\theta_r}_\cong\ar@{->>}[d]_R & W_r(\hat\roi_X^+)/[\zeta_{p^r}]-1\ar@/^2cm/[dd]^{F-\al R} \ar@{->>}[d]^F\\ 
W_{r-1}(\roi_X^{+\flat})/\mu \ar[d]_{1-\al\phi^{-1}}\ar[r]^{\tilde\theta_{r-1}}_\cong & W_{r-1}(\hat\roi_X^+)/[\zeta_{p^{r-1}}]-1\\
W_{r-1}(\roi_X^{+\flat})/\mu \ar[r]^{\tilde\theta_{r-1}}_\cong & W_{r-1}(\hat\roi_X^+)/[\zeta_{p^{r-1}}]-1
}\]
which commutes thanks to the second row of commutative diagrams in \cite[Lem.~3.4]{Bhatt_Morrow_Scholze2}. But the map $1-\al\phi^{-1}$ is an isomorphism, by taking Lemma \ref{lemma_AS_maps}(iii) modulo $p^{r-1}$, whence the result follows from a trivial diagram chase.
\end{proof}

More generally for $j>1$ the following holds:

\begin{corollary}\label{corollary1}
For $j\ge 1$ and $\al\in\xi^j\bb A_\sub{inf}$, the map \[F-\al R:W_r(\hat\roi_X^+)/([\zeta_{p^r}]-1)^j\To W_{r-1}(\hat\roi_X^+)/([\zeta_{p^{r-1}}]-1)^j\] is surjective and has kernel contained inside the kernel of the map \[(\al R)^j:W_r(\hat\roi_X^+)/([\zeta_{p^r}]-1)^j\To W_{r-j}(\hat\roi_X^+)/([\zeta_{p^{r-j}}]-1)^j.\]
\end{corollary}
\begin{proof}
The case $j=1$ follows from the previous lemma. The general case follows by induction on $j>1$ using the following diagram of short exact sequences:
\begin{equation}\xymatrix{
0\ar[r] &W_r(\hat\roi_X^+)/([\zeta_{p^r}]-1)^{j-1}\ar[r]^{[\zeta_{p^r}]-1} \ar[d]_{F-\al'R}&W_r(\hat\roi_X^+)/([\zeta_{p^r}]-1)^j\ar[r]\ar[d]_{F-\al R} & W_r(\hat\roi_X^+)/[\zeta_{p^r}]-1\ar[r] \ar[d]_{F-\al R}& 0\\
0\ar[r] &W_{r-1}(\hat\roi_X^+)/([\zeta_{p^r}]-1)^{j-1}\ar[r]^{[\zeta_{p^{r-1}}]-1} &W_{r-1}(\hat\roi_X^+)/([\zeta_{p^{r-1}}]-1)^j\ar[r] &W_{r-1}(\hat\roi_X^+)/[\zeta_{p^{r-1}}]-1\ar[r]& 0
}\label{eqn_ses}\end{equation}
where $\al':=\al\xi^{-1}$.
\end{proof}

\begin{corollary}\label{corollary2}
Let $i\ge0$ and $j\ge1$.
\begin{enumerate}
\item Given $\al\in\xi^j\bb A_\sub{inf}$ (resp.~$\in\xi^{j+1}\bb A_\sub{inf}$), the map \[F-\al R:(R^i\nu_*W_r(\hat\roi_X^+))/([\zeta_{p^r}]-1)^j\to (R^i\nu_*W_{r-1}(\hat\roi_X^+))/([\zeta_{p^{r-1}}]-1)^j\] becomes injective (resp.~an isomorphism) after applying $\projlimf_{r\sub{ wrt }\al R}$.
\item Given $\al'\in\bb A_\sub{inf}$ (resp.~$\in\xi\bb A_\sub{inf}$), the map \[F-\al'R:(R^i\nu_*W_r(\hat\roi_X^+))[([\zeta_{p^r}]-1)^j]\to (R^i\nu_*W_{r-1}(\hat\roi_X^+))[([\zeta_{p^{r-1}}]-1)^j]\] becomes surjective (resp.~an isomorphism) after applying $\projlimf_{r\sub{ wrt }\al' R}$.
\end{enumerate}
\end{corollary}
\begin{proof}
Fix $\al'\in\bb A_\sub{inf}$ and put $\al=\al'\xi^j$; we prove both parts at once. Taking $R\nu_*$ of the commutative diagram
\begin{equation}\xymatrix@C=2cm{
0\ar[r] & W_r(\hat\roi_X^+)\ar[r]^{([\zeta_{p^r}]-1)^j}\ar[d]^{F-\al' R} & W_r(\hat\roi_X^+)\ar[r]\ar[d]^{F-\al R}& W_r(\hat\roi_X^+)/([\zeta_{p^r}]-1)^j\ar[r]\ar[d]^{F-\al R}& 0\\
0\ar[r] & W_{r-1}(\hat\roi_X^+)\ar[r]^{([\zeta_{p^{r-1}}]-1)^j} & W_{r-1}(\hat\roi_X^+)\ar[r]& W_{r-1}(\hat\roi_X^+)/([\zeta_{p^{r-1}}]-1)^j\ar[r] &0
}\label{eqn_diagram_1}\end{equation}
yields a commutative diagram of short exact sequences:
\begin{equation}\hspace{-1cm}\xymatrix@C=0.5cm{
0\ar[r] & (R^i\nu_*W_r(\hat\roi_X^+))/([\zeta_{p^r}]-1)^j\ar[r]\ar[d]^{F-\al R} & R^i\nu_*(W_r(\hat\roi_X^+)/([\zeta_{p^r}]-1)^j)\ar[r]\ar[d]^{F-\al R}& (R^{i+1}\nu_*W_r(\hat\roi_X^+))[([\zeta_{p^r}]-1)^j]\ar[r]\ar[d]^{F-\al'R}& 0\\
0\ar[r] & R^i\nu_*W_{r-1}(\hat\roi_X^+)/([\zeta_{p^{r-1}}]-1)^j \ar[r] & R^i\nu_*(W_{r-1}(\hat\roi_X^+)/([\zeta_{p^{r-1}}]-1)^j)\ar[r]& (R^{i+1}\nu_*W_{r-1}(\hat\roi_X^+))[([\zeta_{p^{r-1}}]-1)^j]\ar[r] &0
}\label{eqn_diagram_2}\end{equation}
According to Corollary \ref{corollary1}, the right vertical arrow of the first diagram becomes an isomorphism of pro sheaves after applying $\projlimf_{r\sub{ wrt }\al R}$; hence the middle vertical arrow of the second diagram become an isomorphism of pro sheaves after applying $\projlimf_{r\sub{ wrt }\al R}$.

It follows that the left (resp.~right) vertical arrow of the second diagram becomes injective (resp.~surjective) after applying $\projlimf_{r\sub{ wrt }\al R}$ (resp.~$\projlimf_{r\sub{ wrt }\al' R}$).

Now assume that $\al'$ is divisible by $\xi$; to complete the proof it is enough to show that the right vertical arrow of the second diagram becomes injective after applying $\projlimf_{r\sub{ wrt }\al' R}$ (because it then follows that the three vertical arrows become isomorphisms after passing to pro sheaves). Replace $i$ by $i-1$ for simplicity of notation. Since all $[\zeta_{p^r}]-1$-power torsion in $R^{i+1}\nu_*W_r(\hat\roi_X^+)$ is actually killed by $[\zeta_{p^r}]-1$ by Corollary \ref{corollary_torsion_bound}(iii), the inclusion \[R^{i+1}\nu_*W_r(\hat\roi_X^+)[([\zeta_{p^r}]-1)^j]=R^{i+1}\nu_*W_r(\hat\roi_X^+)[[\zeta_{p^r}]-1]\subseteq R^{i+1}\nu_*W_r(\hat\roi_X^+)\] induces an inclusion \[R^{i+1}\nu_*W_r(\hat\roi_X^+)[([\zeta_{p^r}]-1)^j]\subseteq R^{i+1}\nu_*W_r(\hat\roi_X^+)/[\zeta_{p^r}]-1\] which is compatible with $R$ and $F$. But since $\al'$ is now assumed to be divisible by $\xi$, we have already shown that $F-\al'R$ is injective on $\projlimf_{r\sub{ wrt }\al' R}R^{i+1}\nu_*W_r(\hat\roi_X^+)/[\zeta_{p^r}]-1$; hence it is also injective on $\projlimf_{r\sub{ wrt }\al' R}R^{i+1}\nu_*W_r(\hat\roi_X^+)[([\zeta_{p^r}]-1)^j]$, as required.
\end{proof}

Remarkably, we do not know a direct proof of the following useful consequence:

\begin{corollary}\label{corollary3}
Let $i\ge 0$ and $j\ge 1$, and let $\al\in\xi^j\bb A_\sub{inf}$ (resp.~$\in \xi^{j+1}\bb A_\sub{inf}$). Then the map \[F-\al R:W_r\Omega_{\frak X/\roi}^i/([\zeta_{p^r}]-1)^j\To W_{r-1}\Omega_{\frak X/\roi}^i/([\zeta_{p^{r-1}}]-1)^j\] becomes injective (resp.~an isomorphism) after applying $\projlimf_{r\sub{ wrt }\al R}$. 
\end{corollary}
\begin{proof}
There are short exact sequences 
\[0\To R^iv_*W_r(\hat\roi_X^+)[([\zeta_{p^r}]-1)^j]\To R^i\nu_*W_r(\hat\roi_X^+)\To W_r\Omega^i_{\frak X/\roi}\To 0\]
compatible with $\al R$ and $F$ on each term. The sequence remains exact modulo $([\zeta_{p^r}]-1)^j$ since $W_r\Omega^i_{\frak X/\roi}$ is $[\zeta_{p^r}]-1$-torsion-free. The desired claim then follows from Corollary \ref{corollary2}, which asserts that $F-\al R$ acts as an injection on $(R^i\nu_*W_r(\hat\roi_X^+))/([\zeta_{p^r}]-1)^j$ (resp.~an automorphism if $\al\in\xi^{j+1}\bb A_\sub{inf}$) and an automorphism of $\projlimf_{r\sub{ wrt }\al R}R^iv_*W_r(\hat\roi_X^+)[([\zeta_{p^r}]-1)^j]$.
\end{proof}

In the next corollary we apply $\tau^{\le j}$ to the diagrams (\ref{squareF}) and (\ref{squareR}), and then take their difference:

\begin{corollary}\label{corollary_A}
Consider the square
\[\xymatrix{
\tau^{\le j}\tilde{W_r\Omega}_{\frak X}\{j\}\ar[d]\ar[r]^{F-R}&\tau^{\le j}\tilde{W_{r-1}\Omega}_{\frak X}\{j\}\ar[d]\\
\tau^{\le j}R\nu_*W_r(\hat\roi_X^+)\{j\}\ar[r]^{F-R}&\tau^{\le j}R\nu_*W_{r-1}(\hat\roi_X^+)\{j\}
}\]
Then the fibre of the top horizontal arrow is supported in cohomological degrees $\le j$ and, after applying $\projlimf_{r\sub{ wrt }R}$, the induced map to $\tau^{\le j}$ of the fibre of the bottom arrow is an isomorphism of pro sheaves on $\frak X_\sub{\'et}$.
\end{corollary}
\begin{proof}
After trivialising the twists, the square can be rewritten as 
\[\xymatrix{
\tau^{\le j}L\eta_{[\zeta_{p^r}]-1}R\nu_*W_r(\hat\roi_X^+)\ar[d]\ar[r]^{F-\xi^jR}&\tau^{\le j}L\eta_{[\zeta_{p^{r-1}}]-1}R\nu_*W_{r-1}(\hat\roi_X^+)\ar[d]\\
\tau^{\le j}R\nu_*W_r(\hat\roi_X^+)\ar[r]^{F-\xi^jR}&\tau^{\le j}R\nu_*W_{r-1}(\hat\roi_X^+)
},\]
viewed as a pro system with transition maps $\xi^jR$.

Applying $H^j(-)$ to the top horizontal arrows yields $W_r\Omega^j_{\frak X/\roi}\xto{F-R}W_{r-1}\Omega^j_{\frak X/\roi}$, which is surjective by Theorem \ref{theorem_surj_on_formal_schemes}.

By Corollary \ref{corollary_torsion_bound}(ii), the vertical arrows are injective on cohomology, with induced map on cokernels in degree $i$ given by \[(R^i\nu_*W_r(\hat\roi_X^+))/([\zeta_{p^r}]-1)^i\xTo{F-\xi^jR}(R^i\nu_*W_r(\hat\roi_X^+))/([\zeta_{p^r}]-1)^i\] After taking $\projlimf_{r\sub{ wrt }\xi^j R}$ this becomes an isomorphism (resp.~injection) when $i<j$ (resp.~$i=j$) by Corollary \ref{corollary3}. This is enough to deduce the desired result on the homotopy fibres of the horizontal arrows.
\end{proof}

Next we use our results to show that either Breuil--Kisin or Tate twists may be used in the calculation of Frobenius fixed points; here we use the identification $W_r(\hat\roi_X^+)(j)=([\zeta_{p^r}]-1)^jW_r(\hat\roi_X^+)\{j\}$:

\begin{corollary}\label{corollary_B}
The following square of pro-\'etale sheaves becomes homotopy cartesian after applying $\projlimf_{r\sub{ wrt }R}$
\[\xymatrix{
W_r(\hat\roi_X^+)(j)\ar[d]\ar[r]^{F-R}&W_{r-1}(\hat\roi_X^+)(j)\ar[d]\\
W_r(\hat\roi_X^+)\{j\}\ar[r]^{F-R}&W_{r-1}(\hat\roi_X^+)\{j\}
}\]
\end{corollary}
\begin{proof}
The vertical arrows are injective, with induced maps on the cokernels represented by \[F-\xi^jR:W_r(\hat\roi_X^+)/([\zeta_{p^r}]-1)^j\To W_{r-1}(\hat\roi_X^+)/([\zeta_{p^{r-1}}]-1)^j.\] This is an isomorphism of pro sheaves after taking $\projlimf_{r\sub{ wrt }\xi^jR}$ by Corollary \ref{corollary2}.
\end{proof}

\subsection{$p$-torsion-freeness of $p$-adic vanishing cycles}\label{section_end}
Assembling Corollaries \ref{corollary_A} and \ref{corollary_B} yields a natural homotopy fibre sequence \begin{equation}\projlimf_{r}\tau^{\le j}R\nu_*Z_r(j)\To \projlimf_{r\sub{ wrt }R}\tau^{\le j}\tilde{W_r\Omega}_{\frak X}\{j\}\xTo{F-R}\projlimf_{r\sub{ wrt }R}\tau^{\le j}\tilde{W_{r-1}\Omega}_{\frak X}\{j\}.\label{eqn_prelim}\end{equation} To convert this into Theorem \ref{main_theorem_over_C} we need to know that $p$-adic vanishing cycles are $p$-torsion-free; the following result makes this precise:

\begin{proposition}\label{appendix}
For any fixed $j\ge0$, the following are equivalent:
\begin{enumerate}
\item Theorem \ref{main_theorem_over_C} is true for all $N\ge1$ (with $j$ fixed).
\item $R^j\nu_*(\bb Z/p^N\bb Z)\to R^j\nu_*(\bb Z/p^{N-1}\bb Z)$ is surjective for all $N\ge1$.
\item The pro sheaf $\projlimf_rR^{j+1}\nu_*Z_r$ is $p$-torsion-free.
\item The cokernel of $\projlimf_{r\sub{ wrt }\xi R}W_r\Omega^{j}_{\frak X/\roi}\xto{F-\xi R}\projlimf_{r\sub{ wrt }\xi R}W_{r-1}\Omega^{j}_{\frak X/\roi}$ is $p$-torsion-free.
\item The map \[\projlimf_{r\sub{ wrt }\xi R}W_r\Omega^{j}_{\frak X/\roi}/([\zeta_{p^r}]-1,p)\xto{F-\xi R}\projlimf_{r\sub{ wrt }\xi R}W_{r-1}\Omega^{j}_{\frak X/\roi}/([\zeta_{p^{r-1}}]-1,p)\] is injective.
\end{enumerate}
More precisely, the following are naturally isomorphic: the $p$-torsion of the pro sheaf in (iii), the $p$-torsion of the cokernel in (iv), and the kernel in (v).
\end{proposition}

\begin{remark}\label{remark_BK}
Suppose in this remark that $\frak X$ is the $p$-adic completion of a smooth $\roi$-scheme which is defined over the ring of integers of a discretely valued subfield $K\subseteq C$ such that $C=\hat{\res K}$. Adopting the notation of Remark \ref{remark1}(ii), it is known that $\res i^*R\res j_*(\bb Z/p^N\bb Z)\to \res i^*R\res j_*(\bb Z/p^{N-1}\bb Z)$ is surjective for all $N\ge1$: indeed, this follows from the fact that both sides are generated by symbols, by passing to the filtered colimit over all finite extensions of $K$ of S.~Bloch and K.~Kato's classical result \cite[Corol.~6.1.1]{Bloch1986}. Therefore the equivalent conditions of Proposition \ref{appendix} hold in this case and we deduce that Theorem \ref{main_theorem_over_C} is true without needing Lemma \ref{lemma_intersection_of_rou} and Proposition \ref{proposition_p_tf}.
\end{remark}

\begin{proof}[Proof of Proposition \ref{appendix}]
(i)$\Rightarrow$(ii): Taking the cohomology of (\ref{eqn_better}), and trivialising all twists, yields an exact sequence of pro sheaves for each $N\ge1$
\[\projlimf_{r\sub{ wrt }\xi R}W_{r-1}\Omega^{j-1}_{\frak X/\roi}/p^N\To R^j\nu_*(\bb Z/p^N\bb Z)\To\projlimf_{r\sub{ wrt }R}W_r\Omega^{j}_{\frak X/\roi}/p^N\xto{F- R}\projlimf_{r\sub{ wrt } R}W_r\Omega^{j}_{\frak X/\roi}/p^N\]
This maps to the corresponding exact sequence for $p^{N-1}$, whence the snake lemma shows that $R^j\nu_*(\bb Z/p^N\bb Z)\to R^j\nu_*(\bb Z/p^{N-1}\bb Z)$ is surjective if 
\[\ker(W_r\Omega^{j}_{\frak X/\roi}/p^N\xto{F- R}W_r\Omega^{j}_{\frak X/\roi}/p^N)\To \ker(W_r\Omega^{j}_{\frak X/\roi}/p^{N-1}\xto{F- R}W_r\Omega^{j}_{\frak X/\roi}/p^{N-1})\] is surjective; but the surjectivity of this map is a formal consequence of the sheaves $W_r\Omega^j_{\frak X/\roi}$ being $p$-torsion-free and the map $F-R:W_r\Omega^j_{\frak X/\roi}/p\to W_{r-1}\Omega^j_{\frak X/\roi}/p$ being surjective by Theorem \ref{theorem_surj_on_formal_schemes}.

(ii)$\Rightarrow$(iii): Consider the following diagram of sheaves:
\[\xymatrix{
R^j\nu_*Z_r\ar@/_/[rrd]\ar[r] & (R^j\nu_*Z_r)/p^r\ar[r] & R^j\nu_*(Z_r/p^r)\ar[d]& R^j\nu_*(\bb Z/p^r\bb Z)\ar[l]_{\vartheta_r^r}\ar[d]\\
&&R^j\nu_*(Z_r/p)&R^j\nu_*(\bb Z/p\bb Z)\ar[l]_{\vartheta_r^1}
}\]
Our hypothesis implies that the right vertical arrow is surjective. After taking $\projlimf_r$, the two maps $\vartheta$ maps become isomorphisms by Lemma \ref{lemma_p_adic_cyces_via_witt}, and the middle horizontal arrow becomes an isomorphism (since the $p$-power-torsion in each $R^j\nu_*Z_r$ is bounded by Corollary \ref{corollary_torsion_bound}(v)); it follows that $\projlimf_rR^j\nu_*Z_r\to\projlimf_rR^j\nu_*(Z_r/p)$ is surjective, i.e., that $\projlimf_r(R^{j+1}\nu_*Z_r)[p]$ vanishes.

(iii)$\Leftrightarrow$(iv): At line (\ref{eqn_prelim}) we established a homotopy fibre sequence of pro sheaves 
\[\projlimf_{r}\tau^{\le j+1}R\nu_*Z_r(j+1)\To\projlimf_{r\sub{ wrt }R}\tau^{\le j+1}\tilde{W_r\Omega}_{\frak X}\{j+1\}\xTo{F-R}\projlimf_{r\sub{ wrt }R}\tau^{\le j+1}\tilde{W_{r-1}\Omega}_{\frak X}\{j+1\}.\]
Taking its cohomology and trivialising all twists leads to an exact sequence \[\projlimf_{r\sub{ wrt }\xi R}W_r\Omega^{j}_{\frak X/\roi}\xto{F-\xi R}\projlimf_{r\sub{ wrt }\xi R}W_{r-1}\Omega^{j}_{\frak X/\roi}\To \projlimf_rR^{j+1}\nu_*Z_r\To\projlimf_{r\sub{ wrt }R}W_r\Omega^{j+1}_{\frak X/\roi}\xto{F- R}\projlimf_{r\sub{ wrt } R}W_{r-1}\Omega^{j+1}_{\frak X/\roi}\] Since each $W_r\Omega^{j+1}_{\frak X/\roi}$ is $p$-torsion-free, the $p$-torsion of the middle term is the same as the $p$-torsion of the cokernel of the map $F-\xi R$.

(iii)$\Rightarrow$(i): Modding out the homotopy fibre sequence from the start of the previous paragraph (but for $j$ rather than $j+1$) by $p^N$ yields 
\[\projlimf_{r}(\tau^{\le j}R\nu_*Z_r(j))/p^N\To\projlimf_{r\sub{ wrt }R}\tau^{\le j}\tilde{W_r\Omega}_{\frak X}\{j\}/p^N\xTo{F-R}\projlimf_{r\sub{ wrt }R}\tau^{\le j}\tilde{W_{r-1}\Omega}_{\frak X}\{j\}/p^N\] Our assumption that $\projlim_rR^{j+1}\nu_*Z_r$ is $p$-torsion-free implies that the canonical map\linebreak $\projlimf_r(\tau^{\le j}R\nu_*Z_r(j))/p^N\isoto \projlimf_{r}\tau^{\le j}R\nu_*(Z_r(j)/p^N)$ is an isomorphism of pro sheaves; the codomain of this map is moreover isomorphic to $\projlimf_{r}(\tau^{\le j}R\nu_*(\bb Z/p^N(j))$ by Lemma~\ref{lemma_p_adic_cyces_via_witt}, thereby producing the desired homotopy fibre sequence (\ref{eqn_better}).

(iv)$\Leftrightarrow$(v): We consider the commutative diagram
\[\xymatrix@C=2cm{
0\ar[r] & W_r\Omega^{j-1}_{\frak X/\roi}\ar@{->>}[d]^{F-R}\ar[r]^{[\zeta_{p^r}]-1}&W_r\Omega^{j-1}_{\frak X/\roi}\ar[r]\ar[d]^{F-\xi R} & W_r\Omega^{j-1}_{\frak X/\roi}\ar[r]\ar[d]^{F-\xi R}/[\zeta_{p^r}]-1\ar[r]&0\\
0\ar[r] &W_{r-1}\Omega^{j-1}_{\frak X/\roi} \ar[r]_{[\zeta_{p^{r-1}}]-1}&W_{r-1}\Omega^{j-1}_{\frak X/\roi}\ar[r]&W_{r-1}\Omega^{j-1}_{\frak X/\roi}/[\zeta_{p^{r-1}}]-1\ar[r]& 0\\
}\]
in which $F-R$ is surjective by Theorem \ref{theorem_surj_on_formal_schemes}. Therefore the cokernel of the middle vertical arrow is the same as the cokernel of the right vertical arrow, i.e., \[\op{Coker}(W_r\Omega^{j}_{\frak X/\roi}/[\zeta_{p^r}]-1\xto{F-\xi R}W_{r-1}\Omega^{j}_{\frak X/\roi}/[\zeta_{p^{r-1}}]-1).\] But after taking $\projlimf_{r\sub{ wrt }\xi R}$ this map becomes injective by Corollary \ref{corollary3}, and we will prove in a moment that its codomain becomes $p$-torsion-free; the mod $p$ Tor sequence therefore tells us that the $p$-torsion of its cokernel is the same as the kernel of the map mod $p$.
\end{proof}

To complete the proof of the equivalence (iv)$\Leftrightarrow$(v) in the previous proposition we need to know that $\projlimf_{r\sub{ wrt }\xi R}W_r\Omega^j_{\frak X/\roi}/[\zeta_{p^r}]-1$ is $p$-torsion-free; this is a consequence of part (iii) of the following lemma:

\begin{lemma}\label{lemma_p_torsion_free_of_quotients}
Let $r\ge1$.
\begin{enumerate}
\item The sequence \[W_r(\roi)/[\zeta_{p^r}]-1\xto{p^{r-1}}W_r(\roi)/[\zeta_{p^r}]-1\xto{p} W_r(\roi)/[\zeta_{p^r}]-1\] is exact.
\item The map $\xi R: W_r(\roi)/[\zeta_{p^r}]-1\to W_{r-1}(\roi)/[\zeta_{p^{r-1}}]-1$ kills all $p$-torsion in the domain.
\item For each $j\ge 0$, the map \[\xi R:W_{r}\Omega_{\frak X/\roi}^j/[\zeta_{p^{r}}]-1\To W_{r-1}\Omega_{\frak X/\roi}^j/[\zeta_{p^{r-1}}]-1\] kills all $p$-torsion in the domain.
\end{enumerate}
\end{lemma}
\begin{proof}
(i) Since $\roi^\flat$ is a perfect ring of characteristic $p$, the $p$-torsion in $W_r(\roi^\flat)=\bb A_\sub{inf}/p^r$ is given by the multiples of $p^{r-1}$; since $\mu,p$ is a regular sequence $\bb A_\sub{inf}$, it easily follows that the same is true in the ring $\bb A_\sub{inf}/\mu$.  But $\tilde\theta_r$ induces an isomorphism \[W_r(A^\flat)/\mu=\bb A_\sub{inf}/(\tilde\xi_r,\mu)\isoto W_r(\roi)/[\zeta_{p^r}]-1,\] and so the $p$-torsion in the ring $W_r(\roi)/[\zeta_{p^r}]-1$ is also given by multiples of $p^{r-1}$, as required.

(ii) The $p$-torsion in $W_{r}(A)/[\zeta_{p^{r}}]-1$ is given by multiples of $p^{r-1}$ by (i), whose image under $\xi R$ is zero in $W_{r-1}(A)/[\zeta_{p^{r-1}}]-1$ since $p^{r-1}=0$ in this latter ring (as already mentioned in Corollary \ref{corollary_torsion_bound})

(iii) R.~Elkik's results \cite{Elkik1973} imply that $\frak X$ admits an open cover by the formal spectra of rings of the form $\hat R$ (the hat denotes $p$-adic completion), where $R$ is a smooth $\roi$-algebra which receives an \'etale map from a Laurent polynomial algebra $\roi[T_1^{\pm1},\dots,T_d^{\pm1}]$. Since the de Rham--Witt groups base change well under \'etale morphisms \cite[Lem.~10.8]{Bhatt_Morrow_Scholze2}, the problem therefore reduces to showing that the map \[\xi R:W_{r+1}\Omega_{R/\roi}^j/[\zeta_{p^{r+1}}]-1\To W_r\Omega_{R/\roi}^j/[\zeta_{p^r}]-1\] kills all $p$-torsion in the domain, where $R:=\roi[T_1^{\pm1},\dots,T_d^{\pm1}]$ (and we have replaced $r$ by $r+1$ to slightly simplify the following notation).

According to \cite[\S10.4]{Bhatt_Morrow_Scholze2}, there exists an isomorphism of $W_r(\roi)$-modules \[\bigoplus_{a:\{1,\dots,d\}\to\bb Z[\tfrac1p]}\bigoplus_{(I_0,\dots,I_n)\in P_a}V^{u(a)}W_{r-u(a)}(\roi)\Isoto W_r\Omega^j_{R/\roi}\] (where $u(a):=0$ if $a$ is $\bb Z$-valued and $:=-\min_{i=1,\dots,d}\nu_p(a(i))$ else; $V^{u(a)}W_{r-u(a)}(\roi):=0$ if $u(a)\ge r$) given by the sum of certain explicit maps of $W_r(\roi)$-modules \[e(-,a,I_0,\dots,I_n):V^{u(a)}W_{r-u(a)}(\roi)\to W_r\Omega^j_{R/\roi}.\] It is trivial to check from the definition of these maps (which we do not repeat here), that they are compatible with the restriction maps, i.e., $R(e(x,a,I_0,\dots,I_n))=e(Rx,a,I_0,\dots,I_n)$ for all $x\in V^{u(a)}W_{r-u(a)}(\roi)$. This reduces the proof to showing that \[\xi R:V^uW_{r-u}(\roi)/[\zeta_{p^{r}}]-1\To V^uW_{r-1-u}(\roi)/[\zeta_{p^{r-1}}]-1\] kills all $p$-torsion in the domain, for each $u=0,\dots,r-1$. But \[V^u:W_{r-u}(\roi)/[\zeta_{p^{r-u}}]-1\isoto V^uW_{r-u}(\roi)/[\zeta_{p^{r}}]-1\] (and similarly at level $r-1-u$ in place of $r-u$), which reduces the desired result to (ii).
\end{proof}

The previous lemma completes the proof of Proposition \ref{appendix}. To prove Theorem \ref{main_theorem_over_C}, we must now verify one of the conditions of the proposition. In fact, we could simply appeal to \cite[Rem.~10.2]{Bhatt_Morrow_Scholze3}, where the $q$-de Rham complex is used to verify condition (ii). However, we prefer instead to give a proof via de Rham--Witt sheaves in the spirit of this paper; we will require the following lemma, reminiscent of results which already appeared in \cite[\S9.2]{Bhatt_Morrow_Scholze2}.

\begin{lemma}\label{lemma_intersection_of_rou}
For each $r\ge 0$, $j\ge 1$, the canonical map \[W_r\Omega^j_{\frak X/\roi}/([\zeta_{p^r}]-1,p)\To \projlim_{s\ge r}W_r\Omega^j_{\frak X/\roi}/(\tfrac{[\zeta_{p^r}]-1}{[\zeta_{p^s}]-1},p)\] is injective.
\end{lemma}
\begin{proof}
Thanks to a theorem of K.~Kedlaya  \cite{Kedlaya2005}, $\frak X$ admits a cover by the formal spectra of formally smooth $\roi$-algebras which receive a finite \'etale map from $\roi\pid{\ul T^{\pm1}}:=\roi\pid{T_1^{\pm1},\dots,T_d^{1\pm1}}$; see \cite[Lem.~4.9]{Bhatt2016} for further details. It is therefore enough to let $R$ be an arbitrary finite \'etale $\roi\pid{\ul T^{\pm1}}$-algebra and to show that \begin{equation}W_r\Omega^j_{R/\roi}/([\zeta_{p^r}]-1,p)\To \projlim_{s\ge r}W_r\Omega^j_{R/\roi}/(\tfrac{[\zeta_{p^r}]-1}{[\zeta_{p^s}]-1},p)\label{eqn_inj}\end{equation} is injective.

Using the presentation of the de Rham--Witt groups of a Laurent polynomial algebra which has already appeared in the proof of Lemma \ref{lemma_p_torsion_free_of_quotients}(iii), we see that the desired injectivity holds in the case of $\roi\pid{\ul T^{\pm1}}$ if and only if \[M/([\zeta_{p^r}]-1,p)\To \projlim_{s\ge r}M/(\tfrac{[\zeta_{p^r}]-1}{[\zeta_{p^s}]-1},p)\] is injective for each of the $W_r(\roi)$-modules $M=V^{r-i}W_{i}(\roi)$ for $i=1,\dots,r$. Recalling that $V^{r-i}W_{i}(\roi)$ is the same as the $W_r(\roi)$-module $F^{r-i}_*W_i(\roi)$, this can be rewritten \[W_i(\roi)/([\zeta_{p^i}]-1,p)\To \projlim_{s\ge i}W_i(\roi)/(\tfrac{[\zeta_{p^i}]-1}{[\zeta_{p^s}]-1},p).\] Next we use the isomorphism $\theta_i:\roi^\flat/\tfrac{\ep-1}{\ep^{1/p^i}-1}\isoto W_i(\roi)/p$, which sends $\ep^{1/p^m}-1$ to $[\zeta_{p^m}]-1$ for all $m\ge1$, to finally rewrite this as \[\roi^\flat/\ep^{1/p^i}-1\To\projlim_{s\ge i}\roi^\flat/\tfrac{\ep^{1/p^i}-1}{\ep^{1/p^s}-1},\] which is indeed injective by consideration of valuations.

We may now return to our general finite \'etale $\roi\pid{\ul T^{\pm1}}$-algebra $R$. By base changing the just-proved injectivity of \[W_r\Omega^j_{\roi\pid{\ul T^{\pm1}}/\roi}/([\zeta_{p^r}]-1,p)\To \projlim_{s\ge r}W_r\Omega^j_{\roi\pid{\ul T^{\pm1}}/\roi}/(\tfrac{[\zeta_{p^r}]-1}{[\zeta_{p^s}]-1},p)\] along the finite \'etale map $\roi\pid{\ul T^{\pm1}}\to W_r(R)$ (this map is finite \'etale by W.~van der Kallen \cite[\S2]{vanderKallen1986}), and using \'etale base change of the de Rham--Witt complex \cite[Lem.~10.8]{Bhatt_Morrow_Scholze2}, we exactly obtain the desired injectivity of (\ref{eqn_inj}) (note that the base change can be exchanged with the $\projlim_s$ since $\roi\pid{\ul T^{\pm1}}\to W_r(R)$ is finite \'etale).
\end{proof}

Now we verify condition (iii) of Proposition \ref{appendix}:

\begin{proposition}\label{proposition_p_tf}
For each $j\ge 0$, the pro sheaf $\projlim_r R^{j+1}\nu_*Z_r$ is $p$-torsion-free.
\end{proposition}
\begin{proof}
Since $\projlimf_r(R^{j+1}\nu_*Z_r)[p]$ is a quotient of $\projlimf_r R^j\nu_*(Z_r/p)\cong R^j\nu_*(\bb Z/p\bb Z)$, the comparison between (iii) and (v) in Proposition \ref{appendix} gives us an exact sequence of pro sheaves \begin{equation}R^j\nu_*(\bb Z/p\bb Z)\To \projlimf_{r\sub{ wrt }\xi R}W_r\Omega^{j}_{\frak X/\roi}/([\zeta_{p^r}]-1,p)\xto{F-\xi R}\projlimf_{r\sub{ wrt }\xi R}W_{r-1}\Omega^{j}_{\frak X/\roi}/([\zeta_{p^{r-1}}]-1,p).\label{eqn_p_tf}\end{equation} In the rest of this proof we adopt the notation \[\ker_r(F-\xi R):=W_r\Omega^{j}_{\frak X/\roi}/([\zeta_{p^r}]-1,p)\xto{F-\xi R}W_{r-1}\Omega^{j}_{\frak X/\roi}/([\zeta_{p^{r-1}}]-1,p).\] Our goal is to prove that $\projlimf_{r\sub{ wrt }R}\ker_r(F-\xi R)=0$, i.e., for any fixed $r\ge 1$ we must find $r'\ge r$ such that the transition map $(\xi R)^{r'-r}:\ker_{r'}(F-\xi R)\to \ker_r(F-\xi R)$ is zero.

The fact that the left pro sheaf in sequence (\ref{eqn_p_tf}) is constant concretely implies the following: given our fixed $r\ge 1$, there exists $r'\ge r$ such that the image of \[(\xi R)^{s-r}: \ker_s(F-\xi R)\To\ker_r(F-\xi R)\] is the same for every $s\ge r'$. Therefore the image of $(\xi R)^{r'-r}:\ker_{r'}(F-\xi R)\to \ker_r(F-\xi R)$ is equal to $\bigcap_{s\ge r} \Im((\xi R)^{s-r}: \ker_s(F-\xi R)\To\ker_r(F-\xi R))$ and so we need to show that this intersection is zero. Noting that $(\xi R)^{s-r}=\tfrac{[\zeta_{p^r}]-1}{[\zeta_{p^s}]-1}R^{s-r}$, it is therefore enough to check that no non-zero section of $W_r\Omega^{j}_{\frak X/\roi}/([\zeta_{p^r}]-1,p)$ is divisible by $\tfrac{[\zeta_{p^r}]-1}{[\zeta_{p^s}]-1}$ for all $s\ge r$; this is exactly Lemma \ref{lemma_intersection_of_rou}.
\end{proof}

Thanks to Proposition \ref{proposition_p_tf}, we now know that the equivalent conditions of  Proposition \ref{appendix} are all true: in particular, this completes the proof of Theorem \ref{main_theorem_over_C} and of $p$-torsion-freeness of $p$-adic vanishing cycles (Proposition \ref{appendix}(ii)).

\subsection{An extension of $W\Omega_{\frak X/\roi}^j$}\label{subsection_extension}
Here we use Theorem \ref{main_theorem_over_C} to construct an extension of the de Rham--Witt sheaf $W\Omega^j_{\frak X/\roi}$ whose Frobenius-fixed points are precisely $p$-adic vanishing cycles. 

\begin{definition}
Set \[\overline{W\Omega}^j_{\frak X/\roi}:=H^j(\op{Rlim}_{r\sub{ wrt }R}\tau^{\le j}\tilde{W_r\Omega_{\frak X}}\{j\}),\] which is an \'etale sheaf on $\frak X$ equipped with an operator $F:\overline{W\Omega}^j_{\frak X/\roi}\to \overline{W\Omega}^j_{\frak X/\roi}$ induced by the Frobenius maps $F:\tau^{\le j}\tilde{W_r\Omega_{\frak X}}\{j\}\to \tau^{\le j}\tilde{W_{r-1}\Omega_{\frak X}}\{j\}$.
\end{definition}

\begin{theorem}\label{theorem_extension}
\begin{enumerate}
\item For each $j\ge0$ there is a natural exact sequence \[0\To\limone_{r\sub{ wrt }R}W_r\Omega_{\frak X/\roi}^{j-1}\{1\}\To \res{W\Omega}^j_{\frak X/\roi}\To W\Omega_{\frak X/\roi}^j\To 0\] compatible with $F$.
\item For each $N,j\ge 0$ the map $\iota:R^j\nu_*(\bb Z/p^N\bb Z(j))\to W\Omega^j_{\frak X/\roi}/p^N$ lifts to $\iota:R^j\nu_*(\bb Z/p^N\bb Z(j))\to \overline{W\Omega}^j_{\frak X/\roi}/p^N$ and fits into an exact sequence of \'etale sheaves on $\frak X$ \[0\To R^j\nu_*(\bb Z/p^N\bb Z(j))\xTo{\iota} \overline{W\Omega}^j_{\frak X/\roi}/p^N\xto{1-F}\overline{W\Omega}^j_{\frak X/\roi}/p^N\To 0.\]
\end{enumerate}
\end{theorem}
\begin{proof}
Since all cohomology sheaves of $\tau^{\le j}\tilde{W_r\Omega_{\frak X}}\{j\}$ (namely various twists of $W_r\Omega_{\frak X/\roi}^i$, $i\le j$) have vanishing cohomology on affines in the \'etale site of $\frak X$, no derived inverse limits beyond $\projlim^1$ appear in $\op{Rlim}_{r\sub{ wrt }R}\tau^{\le j}\tilde{W_r\Omega_{\frak X}}\{j\}$ and so we obtain a natural short exact sequence \[0\To\limone_{r\sub{ wrt }R}H^{j-1}(\tilde{W_r\Omega_{\frak X}})\To \res{W\Omega}^j_{\frak X/\roi}\To \projlim_{r\sub{ wrt }R}H^{j}(\tilde{W_r\Omega_{\frak X}})\To 0\] Recalling that the outer terms are given respectively by the limits of $W_r\Omega_{\frak X/\roi}^{j-1}\{1\}$ and $W_r\Omega_{\frak X/\roi}^{j}$ respectively gives (i). Moreover, taking the resulting sequence modulo $p^N$ gives an exact sequence \[0\To\limone_{r\sub{ wrt }R}(W_r\Omega_{\frak X/\roi}^{j-1}\{1\}/p^N)\To \res{W\Omega}^j_{\frak X/\roi}/p^N\To W\Omega_{\frak X/\roi}^j/p^N\To 0,\] where modding out by $p^N$ commutes with $\projlim^1$ since the latter functor is right exact (on \'etale sheaves with no higher cohomology on affines); this shows that $\res{W\Omega}^j_{\frak X/\roi}/p^N=H^j(\op{Rlim}_{r\sub{ wrt }R}\tau^{\le j}(\tilde{W_r\Omega_{\frak X}}\{j\}/p^N))$.

Next we compare two copies of (\ref{eqn_better}), firstly for $j-1$ but with an extra Tate twist and secondly for $j$:
\[\hspace{-2cm}\xymatrix{
\projlimf_{r}\tau^{\le j-1}R\nu_*(\bb Z/p^N\bb Z(j))\ar[r]\ar[d]& \projlimf_{r\sub{ wrt }R}\tau^{\le j-1}\tilde{W_r\Omega}_{\frak X}\{j-1\}(1)/p^N\ar[r]^{F-R}\ar[d]&\projlimf_{r\sub{ wrt }R}\tau^{\le j-1}\tilde{W_{r-1}\Omega}_{\frak X}\{j-1\}(1)/p^N\ar[d]\\
\projlimf_{r}\tau^{\le j}R\nu_*(\bb Z/p^N\bb Z(j))\ar[r]& \projlimf_{r\sub{ wrt }R}\tau^{\le j}\tilde{W_r\Omega}_{\frak X}\{j\}/p^N\ar[r]_{F-R}&\projlimf_{r\sub{ wrt }R}\tau^{\le j}\tilde{W_{r-1}\Omega}_{\frak X}\{j\}/p^N
}\]
Note that replacing $\projlimf_{r\sub{ wrt }R}\tau^{\le j}\tilde{W_r\Omega}_\frak X\{j\}$ by the cofiber of the middle vertical arrow does not change $H^j(\op{Rlim}_{r\sub{ wrt }R}(-))$: indeed, since all the cohomologies have vanishing higher \'etale cohomology on affines, only $\projlim$ and $\projlim^1$ appear and so it is enough to check that $H^j(\op{Rlim}_{r\sub{ wrt }R}(-))$ of the top middle term vanishes; but this is $\projlim_{r\sub{ wrt }R}^1W_r\Omega^{j-1}_{\frak X/\roi}(1)/p^N$, which vanishes since the restriction maps are surjective in this system.

This argument (replacing $j$ by $j+1$) also shows that $H^{j+1}(\op{Rlim}_{r\sub{ wrt }R}(-))$ of the bottom middle term vanishes. It now follows at once that $H^j(\op{Rlim}_{r\sub{ wrt }R}(-))$ of the vertical cofibers of the diagram is exactly the desired short exact sequence (ii).
\end{proof}

Similar arguments to those of the previous proof also lead to a description of the kernel of the map $\iota$ as the cokernel of a twisted $F-R$ map:

\begin{corollary}\label{lemma_5_term}
For each $N,j\ge0$ there is a natural exact sequence of \'etale sheaves on $\frak X$
\[\hspace{-1.8cm}\xymatrix@C=5mm{
0\ar[r]& \projlimf_{r\sub{ wrt }R}W_{r}\Omega^{j-1}_{\frak X/\roi}\{1\}/([\zeta_{p^r}]-1, p^N)\ar[r]^-{F-R}&  \projlimf_{r\sub{ wrt }R}W_{r-1}\Omega^{j-1}_{\frak X/\roi}\{1\}/([\zeta_{p^{r-1}}]-1, p^N)
\ar `r/8pt[d] `/12pt[l] `^dl[ll] `^r/0pt[dll] [dll]\\
 R^j\nu_*(\bb Z/p^N\bb Z(j))\ar[r]^{\iota} & \projlimf_{r\sub{ wrt }R}W_r\Omega^{j}_{\frak X/\roi}/p^N\ar[r]^{F- R} & \projlimf_{r\sub{ wrt } R}W_r\Omega^{j}_{\frak X/\roi}/p^N\ar[r] & 0
}\]
\end{corollary}
\begin{proof}
The obstruction to injectivity of the map $\iota$ in Corollary \ref{corollary_les} is given by the cokernel of $F-R:\projlim_{r\sub{ wrt }R}W_r\Omega^{j-1}_{\frak X/\roi}\{1\}/p^N\to \projlim_{r\sub{ wrt }R}W_{r-1}\Omega^{j-1}_{\frak X/\roi}\{1\}/p^N$ which, trivialising the twists, can be rewritten as $F-\xi R:\projlim_{r\sub{ wrt }\xi R}W_r\Omega^{j-1}_{\frak X/\roi}/p^N\to \projlim_{r\sub{ wrt }\xi R}W_{r-1}\Omega^{j-1}_{\frak X/\roi}/p^N$. But, as we already saw in the proof of (iv)$\Leftrightarrow$(v) in Proposition \ref{appendix}, this latter cokernel is unchanged if we mod out by the elements $[\zeta_{p^r}]-1$, after which $F-\xi R$ becomes injective by the fact that condition (v) of Proposition \ref{appendix} is now known to be true (an easy induction to extend the result to $p^N$ rather than only $p$ is also required, using the $p$-torsion-freeness of $\projlimf_{r\sub{ wrt }R}W_{r}\Omega^{j-1}_{\frak X/\roi}\{1\}/[\zeta_{p^r}]-1$ proved in Lemma \ref{lemma_p_torsion_free_of_quotients}(iii)). Replacing the twists gives the claimed exact sequence.
\end{proof}

\subsection{Frobenius-fixed points of $\bb A\Omega_\frak X$}\label{subsection_nygaard}
Here we show how Theorem \ref{main_theorem_over_C} leads to an interpretation of $p$-adic vanishing cycles as the Frobenius-fixed points of $\bb A\Omega_\frak X=L\eta_\mu R\nu_*\bb A_{\sub{inf},X}\simeq\op{Rlim}_{r\sub{ wrt }F}\tilde{W_r\Omega}_\frak X$, the $q$-de Rham complex over $\bb A_\sub{inf}$ which was constructed in \cite{Bhatt_Morrow_Scholze2}. We stress that stronger such statements may be proved while simultaneously avoiding any de Rham--Witt complexes; see \cite[\S10]{Bhatt_Morrow_Scholze3}. Therefore we will be brief here.

We will use the Breuil--Kisin--Fargues twist over $\bb A_\sub{inf}$ \cite[\S4.3]{Bhatt_Morrow_Scholze2}, namely the invertible $\bb A_\sub{inf}$-module \[\bb A_\sub{inf}\{j\}:=\projlim_{r\sub{ wrt }F}W_r(\roi)\{j\}\cong\tfrac1\mu\bb A_\sub{inf}(j)\] for each $j\ge0$; this is anti-effective as a Breuil--Kisin--Fargues module in the sense that it is equipped with an inverse Frobenius semi-linear endomorphism $\phi^{-1}$ (induced by the maps $R:W_r(\roi)\{j\}\to W_{r-1}(\roi)\{j\}$), which is injective and has image $\xi^j\bb A_\sub{inf}\{j\}$. Recalling that $\mu=\xi\phi^{-1}(\mu)$ and that therefore $L\eta_\mu=L\eta_\xi L\eta_{\phi^{-1}(\mu)}$, it is easy to mimic arguments in Section \ref{section_bk2} to check that there is an  induced inverse Frobenius $\phi^{-1}:\tau^{\le j}\bb A\Omega_\frak X\{j\}\to \tau^{\le j}\bb A\Omega_\frak X\{j\}$;
The following proof implicitly extends this modulo $p^N$ from $(\tau^{\le j}\bb A\Omega_\frak X\{j\})/p^N$ to $\tau^{\le j}(\bb A\Omega_\frak X\{j\}/p^N)$, using which we prove:
 
\begin{theorem}\label{theorem_Nygaard}
For each $N,j\ge0$ there is a natural fibre sequence of complexes of sheaves on $\frak X_\sub{\'et}$ \[\tau^{\le j}R\nu_*(\bb Z/p^N\bb Z(j))\To \tau^{\le j}(\bb A\Omega_{\frak X}\{j\}/p^N)\xto{1-\phi^{-1}}\tau^{\le j}(\bb A\Omega_{\frak X}\{j\}/p^N)\]
\end{theorem}
\begin{proof}
Theorem \ref{main_theorem_over_C} remains valid for formal reasons if $\projlimf_{r\sub{ wrt }R}$ is replaced by $\projlimf_{r\sub{ wrt }F}$ (since the two possible transition maps $F,R$ agree on the fibre of $F-R$). By then taking $\op{Rlim}_{r\sub{ wrt }F}$ we obtain a fibre sequence 
\[\tau^{\le j}R\nu_*(\bb Z/p^N\bb Z(j))\To \op{Rlim}_{r\sub{ wrt }F}\tau^{\le j}(\tilde{W_r\Omega}\{j\}/p^N)\xto{F-R} \op{Rlim}_{r\sub{ wrt }F}\tau^{\le j}(\tilde{W_r\Omega}\{j\}/p^N)\] But $\bb A\Omega_{\frak X}\{j\}/p^N=\op{Rlim}_{r\sub{ wrt }F}(\tilde{W_r\Omega}_\frak X\{j\}/p^N)$, so applying $\tau^{\le j}$ to the middle and right terms of the above complex give exactly $\tau^{\le j}(\bb A\Omega_{\frak X}\{j\}/p^N)$.
\end{proof}

\section{Addendum}\label{addendum}

\subsection{Comments on the map $\iota$}\label{subsection_iota}
We expect that the map $\iota:R^j\nu_*(\bb Z/p^N\bb Z(j))\to W_r\Omega^j_{\frak X/\roi}/p^N$ of Corollary \ref{corollary_les} is the $\op{dlog}$ map defined through viewing $R^j\nu_*(\bb Z/p^N\bb Z(j))$ as symbols (see Section \ref{section_Geisser_Hesselholt} for further discussion), but since we cannot currently offer a new proof of Bloch--Kato's symbolic generation of $p$-adic vanishing cycles, we ignore the question for the time being. Assuming that it is true, we present here an informal calculation indicating, in particular, that certain $\projlim^1$ terms cannot be discarded from our study; to summarise what follows, $R^j\nu_*(\bb Z/p^N\bb Z(j))$ appears to be separated, but not complete, with respect to a certain ramification filtration defined as the kernels of the $\iota$ maps.

More precisely, by explicitly manipulating symbols in de Rham--Witt groups, it seems that the principal units killed by \[\dlog:\roi_\frak X^\times\to W_r\Omega^1_{\frak X/\roi}/p,\qquad f\mapsto \tfrac{d[f]}{[f]}\] are those which are $\equiv 1$ modulo $p\tfrac{\zeta_p-1}{\zeta_{p^r}-1}$; conversely, any principle unit which is $\equiv 1$ modulo $p(\zeta_p-1)$ is a $p^\sub{th}$-power \'etale locally. Since \begin{equation}\roi_\frak X/p(\zeta_p-1)\To \projlim_r\roi_\frak X/p\tfrac{\zeta_p-1}{\zeta_{p^r}-1}\label{eqn_inj_not_surj}\end{equation} is injective but not surjective, we see that the above $\dlog$ map is not injective for any fixed value of $r$, but that at least it is after taking the limit over $r$, i.e., $\dlog: \roi_\frak X^\times/p\into W\Omega^1_{\frak X/\roi}/p$. More generally, $\iota$ should not be injective for any fixed value of $r$, but probably $\iota:R^j\nu_*(\bb Z/p^N\bb Z(j))\to W\Omega^j_{\frak X/\roi}/p^N$ is injective.

In light of this it is natural to ask whether the extension $\res{W\Omega}^j_{\frak X/\roi}$ of ${W\Omega}^j_{\frak X/\roi}$ from Section \ref{subsection_extension} is really necessary, i.e., can Theorem \ref{theorem_extension}(ii) be improved to an exact sequence \[0\To R^j\nu_*(\bb Z/p^N\bb Z(j))\xTo{\iota} {W\Omega}^j_{\frak X/\roi}/p^N\xto{1-F}{W\Omega}^j_{\frak X/\roi}/p^N\To 0\,?\] (Equivalently, is the endomorphism $F-R$ of $\projlim_{r\sub{ wrt }R}^1W_r\Omega_{\frak X/\roi}^{j-1}\{1\}$ surjective?) But this holds if and only if $R^j\nu_*(\bb Z/p^N\bb Z(j))$ is complete for the filtration $\ker(R^j\nu_*(\bb Z/p^N\bb Z(j))\xto{\iota} W_r\Omega^j_{\frak X/\roi}/p^N)$, $r\ge1$. The non-surjectivity of (\ref{eqn_inj_not_surj}) suggests that this is not the case.

\subsection{Relation to the result of Geisser--Hesselholt}\label{section_Geisser_Hesselholt}
We finish the article with a discussion of the related result of Geisser and Hesselholt \cite{GeisserHesselholt2006c}. Given a smooth $\roi_K$-scheme $\cal X$, where $K$ is a complete discrete valuation field of mixed characteristic with perfect residue field, they construct (under the assumption $\zeta_{p^N}\in K$) a short exact sequence of pro sheaves on the \'etale site of the special fibre of $\cal X$: \[0\To i^*R^jj_*(\bb Z/p^N\bb Z(j))\xTo{\op{dlog}} \projlimf_{r\sub{ wrt R}}W_r\Omega^j_{(\cal X,M_\cal X)}/p^N\xto{F-R}\projlimf_{r\sub{ wrt R}}W_{r-1}\Omega^j_{(\cal X,M_\cal X)}/p^N\To 0\] (where $i,j$ denote the usual inclusion of the special and generic fibres; the unfortunate double usage of $j$ should not cause confusion). Here $W_r\Omega^\blob_{(\cal X,M_\cal X)}$ denotes the absolute de Rham--Witt complex of $\cal X$ as a log scheme (with respect to the usual log structure given by the special fibre), and $\op{dlog}$ is induced from the $\op{dlog}$ map on Milnor $K$-theory of the sheaf $\roi_\cal X[\tfrac1p]$ via the diagram
\[\xymatrix@C=5cm{
K_j^M(\roi_\cal X[\tfrac1p])/p^N\ar@{->>}[d]_{\text{Bloch--Kato's symbol map}}\ar[r]^{\{f_1,\dots,f_j\}\mapsto\op{dlog}[f_1]\wedge\cdots\wedge\dlog[f_j]} & W_r\Omega^j_{(\cal X,M_\cal X)}/p^N\\
i^*R^jj_*(\bb Z/p^N\bb Z(j))\ar[ur]_{\op{dlog}} & 
}\]
Geisser--Hesselholt show that the $\op{dlog}$ map on Milnor $K$-theory really does descend to $i^*R^jj_*(\bb Z/p^N\bb Z(j))$ and establish their short exact sequence by comparing the graded pieces of the ramification filtration on $i^*R^jj_*(\bb Z/p^N\bb Z(j))$ with an explicit filtration on $W_r\Omega^j_{(\cal X,M_\cal X)}/p^N$ defined in terms of multiples of the maximal ideal of $\roi_K$.

Let us now adopt the framework of the current article with $C:=\hat{\res K}$ and $\frak X:=$the $p$-adic completion of $\cal X\times_{\roi_K}\roi_{\res K}$. There is a canonical map $W_r\Omega^j_{(\cal X,M_\cal X)}/p^N\to W_r\Omega^j_{\frak X/\roi}/p^N$ which becomes an isomorphism when we take the filtered colimit of the domain over all finite extensions of $K$: this is an easy consequence of the facts that de Rham--Witt complexes commute with filtered colimits, that the absolute de Rham--Witt group $W_r\Omega^j_\roi$ is $p$-divisible for $j\ge1$,  and that the extra class $\dlog[\pi_K]$ in the de Rham--Witt group of the log scheme $\cal X$ becomes divisible by $p^N$ after passing to a suitable finite extension of $K$. (See also \cite[Add.~1.3.6]{GeisserHesselholt2006c} and note that the boundary map there becomes zero after extracting a $p^\sub{th}$ root of $\pi_K$.)

In particular, assuming that the map $\iota:R^j\nu_*(\bb Z/p^N\bb Z(j))\to W_r\Omega^j_{\frak X/\roi}/p^N$ really is $\op{dlog}$, then we obtain a commutative diagram
\[\hspace{-0.5cm}\xymatrix{
\cdots\ar[r]& \res i^*R^j\res j_*(\bb Z/p^N\bb Z(j))\ar[r]^-{\iota}&\projlimf_{r\sub{ wrt }R}W_r\Omega^{j}_{\frak X/\roi}/p^N\ar[r]^{F- R}&\projlimf_{r\sub{ wrt } R}W_r\Omega^{j}_{\frak X/\roi}/p^N\ar[r]&0\\
0\ar[r]& i^*R^jj_*(\bb Z/p^N\bb Z(j))\ar[u]\ar[r]_-{\op{dlog}}& \projlimf_{r\sub{ wrt R}}W_r\Omega^j_{(\cal X,M_\cal X)}/p^N\ar[u]\ar[r]_{F-R}&\projlimf_{r\sub{ wrt R}}W_{r-1}\Omega^j_{(\cal X,M_\cal X)}/p^N\ar[u]\ar[r]& 0
}\]
The necessary value of $r$ to ensure injectivity of $\dlog: i^*R^jj_*(\bb Z/p^N\bb Z(j))\to W_r\Omega^j_{(\cal X,M_\cal X)}/p^N$ depends on the ramification degree of $K$ \cite[Rmk.~2.1.6]{GeisserHesselholt2006c}, offering still further evidence that $\iota$ is not injective for any fixed value of $r\ge 1$.

Although the bottom row of the previous diagram remains valid after replacing $K$ by any finite extension, the resulting filtered colimit cannot be exchanged with the pro sheaves. We therefore know of no way to interpret the top row as a filtered colimit of Geisser--Hesselholt's result over all finite extensions of $K$.

\begin{appendix}
\section{Surjectivity of $F-1$ in the \'etale topology}
We prove here that $F-R$ is surjective on the relative de Rham--Witt sheaves of a wide class of formal schemes, including any scheme on which $p$ is nilpotent. Fix throughout this appendix a $\bb Z_{(p)}$-algebra $A$.

We begin with some useful notation, supposing that $B$ is an $A$-algebra and that $I\subseteq B$ is an ideal. Set
\begin{align*}
W_r\Omega_{(B,I)/A}^j&:=\ker(W_r\Omega^j_{B/A}\To W_{r}\Omega^j_{(B/I)/A}),\\ \\
W_r\Omega_{B/A}^{j,F=1}&:=\ker(W_r\Omega_{B/A}^j\xto{F-R}W_{r-1}\Omega_{B/A}^j),\\ \\
W_r\Omega_{(B,I)/A}^{j,F=1}&:=\ker(W_r\Omega_{(B,I)/A}^j\xto{F-R}W_{r-1}\Omega_{(B,I)/A}^j)\\
&\phantom{:}=\ker(W_r\Omega_{B/A}^{j,F=1}\To W_{r}\Omega_{(B/I)/A}^{j,F=1}
\end{align*}

\begin{lemma}\label{lemma_app1}
Let $B$ be an $A$-algebra and $I\subseteq A$ an ideal. Then $W_r\Omega_{(B,I)/A}^\blob$ is the dg ideal of the dg algebra $W_r\Omega_{B/A}^\blob$ generated by $W_r(I)\subseteq W_r(B)$
\end{lemma}
\begin{proof}
Standard identities for the operators on the relative de Rham--Wtt complex show that $R,F,V$ all descend to $W_r\Omega_{B/A}^\blob$ modulo the stated dg ideal, thereby yielding a $F$-$V$-procomplex for $A\to B/I$ which clearly maps surjectively to $W_r\Omega^\blob_{(B/I)/A}$. The universal property of the relative de Rham--Witt complex for $A\to B/I$ yields a section in the other direction, whence both maps are isomorphisms. See \cite[Lem.~2.4]{GeisserHesselholt2006b} for the details of the argument (which are only stated in the case $A=\bb F_p$ but work verbatim in general).
\end{proof}

\begin{lemma}\label{lemma_surj_of_1-F_in_nilp_case}
Let $B$ be an $A$-algebra and $I\subseteq J\subseteq B$ ideals such that $I$ is nilpotent; fix $j\ge0$ and $r\ge 1$. Then:
\begin{enumerate}
\item The map $R-F:W_{r}\Omega_{(B,I)/A}^j\to W_{r-1}\Omega_{(B,I)/A}^j$ is surjective.
\item The canonical maps $W_r\Omega_{B/A}^{j,F=1}\to W_r\Omega_{(B/I)/A}^{j,F=1}$ and $W_r\Omega_{(B,J)/A}^{j,F=1}\to W_r\Omega_{(B/I,J/I)/A}^{j,F=1}$ are surjective.
\end{enumerate}
\end{lemma}
\begin{proof}
These assertions were proved in \cite[Lem.~2.9]{Morrow_pro_GL} in the case that $A=\bb F_p$, but the given arguments work in general using Lemma \ref{lemma_app1}.
\end{proof}

\begin{corollary}\label{proposition_limit_of_F-fixed_points}
Let $B$ be an $A$-algebra and $I\subseteq J\subseteq B$ ideals; fix $j\ge0$ and $r\ge 1$. Then:
\begin{enumerate}
\item The canonical maps $\projlim_sW_r\Omega_{(B/I^s)/A}^{j,F=1}\to W_r\Omega_{(B/I)/A}^{j,F=1}$ and $\projlim_sW_r\Omega_{(B/I^s,J/I^s)/A}^{j,F=1}\to W_r\Omega_{(B/I,J/I)}^{j,F=1}$ are surjective.
\item The map $R-F:\projlim_sW_r\Omega_{(B/I^s,I/I^s)/A}^j\to \projlim_sW_{r-1}\Omega_{(B/I^s,I/I^s)/A}^j$ is surjective.
\end{enumerate}
\end{corollary}
\begin{proof}
(i): The transition maps in the systems $\{W_r\Omega_{(B/I^s)/A}^{j,F=1}\}_s$ and $\{W_r\Omega_{(B/I^s,J/I^s)/A}^{j,F=1}\}_s$ are surjective, by Lemma \ref{lemma_surj_of_1-F_in_nilp_case}(ii).

(ii): Lemmas \ref{lemma_surj_of_1-F_in_nilp_case} yields short exact sequences \[0\To W_r\Omega_{(B/I^s,I/I^s)/A}^{j,F=1}\To W_r\Omega_{(B/I^s,I/I^s)/A}^j\xto{F-R}W_{r-1}\Omega_{(B/I^s,I/I^s)/A}^j\To0,\] where the transition maps over $s\ge1$ are surjective on the left. Taking $\projlim_s$ completes the proof.
\end{proof}

Now suppose that $Y$ is an arbitrary $A$-scheme. Then the sheafification of $Y_\sub{\'et}\ni U\mapsto W_r\Omega^j_{\roi_Y(U)/A}$ is an \'etale sheaf $W_r\Omega^j_{Y/A}$ with the property that, for each affine $\Spec B\in Y_\sub{\'et}$, there exists a natural isomorphism $W_r\Omega^j_{Y/A}(\Spec B)\cong W_r\Omega^j_{B/A}$; indeed, this follows from flat descent and that facts that if $B\to B'$ is an \'etale morphism of $A$-algebras, then $W_r(B)\to W_r(B')$ is \'etale and $W_r\Omega^j_{B/A}\otimes_{W_r(B)}W_r(B')\isoto W_r\Omega^j_{B'/A}$ \cite[Thm.~10.4, Lem.~10.8]{Bhatt_Morrow_Scholze2}. Also set \[W_r\Omega^{j,F=1}_{Y/A}:=\ker(W_r\Omega_{Y/A}^j\xto{F-R}W_{r-1}\Omega_{Y/A}^j),\] whose sections on any $\Spec B\in Y_\sub{\'et}$ are $W_r\Omega^{j,F=1}_{B/A}$. Similarly, given a closed subscheme $i:Z\into Y$, the \'etale sheaves
\begin{align*}W_r\Omega^j_{(Y,Z)/A}&:=\ker(W_r\Omega^j_{Y/A}\To i_*W_r\Omega^j_{Z/A})\\
W_r\Omega^{j,F=1}_{(Y,Z)/A}&:=\ker(W_r\Omega_{(Y,Z)/A}^j\xto{F-R}i_*W_{r-1}\Omega_{(Y,Z)/A}^j)\\
&\phantom{:}=\ker(W_r\Omega^{j,F=1}_{Y/A}\To i_*W_r\Omega^{j,F=1}_{Z/A})
\end{align*}
have the property that, on any affine $\Spec B\in Y_\sub{\'et}$ with associated ideal $I\subseteq B$ defining $\Spec B\times_ZY$, their sections are given respectively by $W_r\Omega^j_{(B,I)/A}$ and $W_r\Omega^{j,F=1}_{(B,I)/A}$.

Finally, suppose that $\cal Y$ is an arbitrary formal $A$-scheme; to be precise (to avoid any confusion about the theory of non-Noetherian formal schemes), we ask only that $\cal Y=\indlimf_s\cal Y_s$ be an ind $A$-scheme such that each transition map $\cal Y_s\to \cal Y_{s+1}$ is a closed embedding defined by a nilpotent ideal sheaf. We may therefore identify the \'etale sites of $\cal Y_1,\cal Y_2,\dots$ -- denote this common site by $\cal Y_\sub{\'et}$ -- and  define a sheaf on $\cal Y_\sub{\'et}$ by $W_r\Omega^j_{\cal Y/A}:=\projlim_sW_r\Omega^j_{\cal Y_s/A}$.

\begin{theorem}\label{theorem_surj_on_formal_schemes}
Under the set-up of the previous paragraph, assume moreover that $p$ is nilpotent on $\cal Y_1$. Then the map $F-R:W_r\Omega^j_{\cal Y/A}\to W_{r-1}\Omega^j_{\cal Y/A}$ of sheaves on $\cal Y_\sub{\'et}$ is surjective for each $j\ge 0$.
\end{theorem}
\begin{proof}
After replacing $\cal Y_1$ by $\cal Y_1\times_AA/pA$ we may suppose that $p=0$ on $\cal Y_1$; this will save us from needing to repeat an argument.

Recall again that $\cal Y_s\into \cal Y_{s+1}$ induces an isomorphism of \'etale sites, and that moreover any given object $U\in \cal Y_{s+1,\sub{\'et}}$ is affine if and only if its pullback $U\times_{Y_{s+1}}\cal Y_s\in \cal Y_{s,\sub{\'et}}$ is affine. Thus a basis for $\cal Y_\sub{\'et}$ is given by those $U\in\cal Y_\sub{\'et}$ corresponding to affines in $\cal Y_{s,\sub{\'et}}$ for any (equivalently, all) $s\ge1$. Let $U$ be such a basis element, with corresponding affines $\Spec B_s\in \cal Y_{s,\sub{\'et}}$ for each $s\ge1$. Note that $B_{s}\to B_{s-1}$ is a surjection with nilpotent kernel; let $I_s:=\ker(B_s\to B_1)$.

For each $s\ge1$, we consider the following diagram of abelian groups
\[\xymatrix@C=1.3cm@R=0.5cm{
&0\ar[d]&0\ar[d] & 0\ar[d]\\
0\ar[r]&W_r\Omega_{(B_s,I_s)/A}^{j,F=1}\ar[r]\ar[d]&W_r\Omega_{(B_s,I_s)/A}^j\ar[d]\ar[r]^{F-R}&W_{r-1}\Omega_{(B_s,I_s)/A}^j\ar[d]\ar[r]&0\\
0\ar[r]&W_r\Omega_{B_s/A}^{j,F=1}\ar[d]\ar[r]&W_r\Omega_{B_s/A}^j\ar[d]\ar[r]^{F-R}&W_{r-1}\Omega_{B_s/A}^j\ar[d]&\\
0\ar[r]&W_r\Omega_{B_1/A}^{j,F=1}\ar[d]\ar[r]&W_r\Omega_{B_1/A}^j\ar[d]\ar[r]^{F-R}&W_{r-1}\Omega_{B_1/A}^j\ar[d]&\\
&0&0&0
}\]
in which the top row and left column are exact by Lemma \ref{lemma_surj_of_1-F_in_nilp_case}; that lemma also tells us that each transition map in the top left corner, i.e., $W_r\Omega_{(B_{s},I_s)/A}^{j,F=1}\to W_r\Omega_{(B_{s-1},I_{s-1})/A}^{j,F=1}$, is surjective. There is therefore no $\projlim^1$ obstruction and so taking $\projlim_s$ yields a diagram with exact rows and columns
\[\xymatrix@C=1.3cm@R=0.5cm{
&0\ar[d]&0\ar[d] & 0\ar[d]\\
0\ar[r]&\projlim_sW_r\Omega_{(B_s,I_s)/A}^{j,F=1}\ar[r]\ar[d]&\projlim_sW_r\Omega_{(B_s,I_s)/A}^j\ar[d]\ar[r]^{F-R}&\projlim_sW_{r-1}\Omega_{(B_s,I_s)/A}^j\ar[d]\ar[r]&0\\
0\ar[r]&\projlim_sW_r\Omega_{B_s/A}^{j,F=1}\ar[d]\ar[r]&\projlim_sW_r\Omega_{B_s/A}^j\ar[d]\ar[r]^{F-R}&\projlim_sW_{r-1}\Omega_{B_s/A}^j\ar[d]&\\
0\ar[r]&\projlim_sW_r\Omega_{B_1/A}^{j,F=1}\ar[d]\ar[r]&\projlim_sW_r\Omega_{B_1/A}^j\ar[d]\ar[r]^{F-R}&\projlim_sW_{r-1}\Omega_{B_1/A}^j\ar[d]&\\
&0&0&0
}\]
The bottom right square in this diagram is precisely the sections of
\[\xymatrix{
W_r\Omega^j_{\cal Y/A}\ar[r]^{F-R}\ar[d] & W_{r-1}\Omega^j_{\cal Y/A}\ar[d] \\
W_r\Omega^j_{\cal Y_1/A}\ar[r]^{F-R} & W_{r-1}\Omega^j_{\cal Y_1/A}
}\]
on the arbitrary basis element $U$, and so we have proved that the induced map $F-R:W_r\Omega^j_{(\cal Y,\cal Y_1)/A}\to W_{r-1}\Omega^j_{(\cal Y,\cal Y_1)/A}$ (the notation is the obvious one) is surjective. To complete the proof, it is therefore necessary and sufficient to show that \[F-R:W_r\Omega^j_{\cal Y_1/A}\To W_{r-1}\Omega^j_{\cal Y_1/A}\] is surjective.

But $Y:=\cal Y_1$ is a $\bb F_p$-scheme, so there is a natural surjection of de Rham--Witt sheaves $W_r\Omega^j_{Y/\bb F_p}\to W_r\Omega^j_{Y/A}$ (note that the latter is the same as $W_r\Omega^j_{Y/(A/pA)}$) and it is therefore enough to show that \[F-R:W_r\Omega^j_{Y/\bb F_p}\To W_{r-1}\Omega^j_{Y/\bb F_p}\] is surjective. It is clearly enough to consider the case in which $Y$ is affine and of finite type over $\bb F_p$, in which case we may pick a closed embedding into a smooth affine $\bb F_p$-scheme and appeal to the surjectivity in the smooth case \cite[Prop.~II.3.26]{Illusie1979}.
\end{proof}

\end{appendix}

\bibliographystyle{acm}
\bibliography{../Bibliography}

\def\cprime{$'$}
\begin{thebibliography}{10}

\bibitem{Bhatt2016}
{\sc Bhatt, B.}
\newblock Specializing varieties and their cohomology from characteristic $0$
  to characteristic $p$.
\newblock {\em {\tt arXiv:1606.01463}\/} (2016).

\bibitem{Bhatt_Morrow_Scholze2}
{\sc Bhatt, B., Morrow, M., and Scholze, P.}
\newblock Integral $p$-adic {H}odge theory.
\newblock {\em {\tt arXiv:1602.03148}\/} (2016).

\bibitem{Bhatt_Morrow_Scholze3}
{\sc Bhatt, B., Morrow, M., and Scholze, P.}
\newblock Topological {H}ochschild homology and integral $p$-adic {H}odge
  theory.
\newblock {\em {\tt arXiv} preprint\/} (2017).

\bibitem{Bloch1986}
{\sc Bloch, S., and Kato, K.}
\newblock {$p$}-adic \'etale cohomology.
\newblock {\em Inst. Hautes \'Etudes Sci. Publ. Math.}, 63 (1986), 107--152.

\bibitem{Elkik1973}
{\sc Elkik, R.}
\newblock Solutions d'\'equations \`a coefficients dans un anneau hens\'elien.
\newblock {\em Ann. Sci. \'Ecole Norm. Sup. (4) 6\/} (1973), 553--603 (1974).

\bibitem{GabberRamero2003}
{\sc Gabber, O., and Ramero, L.}
\newblock {\em Almost ring theory}, vol.~1800 of {\em Lecture Notes in
  Mathematics}.
\newblock Springer-Verlag, Berlin, 2003.

\bibitem{GeisserHesselholt2006c}
{\sc Geisser, T., and Hesselholt, L.}
\newblock The de {R}ham-{W}itt complex and {$p$}-adic vanishing cycles.
\newblock {\em J. Amer. Math. Soc. 19}, 1 (2006), 1--36 (electronic).

\bibitem{GeisserHesselholt2006b}
{\sc Geisser, T., and Hesselholt, L.}
\newblock On the {$K$}-theory of complete regular local {$\Bbb F_p$}-algebras.
\newblock {\em Topology 45}, 3 (2006), 475--493.

\bibitem{GeisserLevine2000}
{\sc Geisser, T., and Levine, M.}
\newblock The {$K$}-theory of fields in characteristic {$p$}.
\newblock {\em Invent. Math. 139}, 3 (2000), 459--493.

\bibitem{Huber1996}
{\sc Huber, R.}
\newblock {\em \'{E}tale cohomology of rigid analytic varieties and adic
  spaces}.
\newblock Aspects of Mathematics, E30. Friedr. Vieweg \& Sohn, Braunschweig,
  1996.

\bibitem{Illusie1979}
{\sc Illusie, L.}
\newblock Complexe de de\thinspace {R}ham-{W}itt et cohomologie cristalline.
\newblock {\em Ann. Sci. \'Ecole Norm. Sup. (4) 12}, 4 (1979), 501--661.

\bibitem{Kedlaya2005}
{\sc {Kedlaya}, K.~S.}
\newblock {More \'etale covers of affine spaces in positive characteristic.}
\newblock {\em {J. Algebr. Geom.} 14}, 1 (2005), 187--192.

\bibitem{LangerZink2004}
{\sc Langer, A., and Zink, T.}
\newblock De {R}ham-{W}itt cohomology for a proper and smooth morphism.
\newblock {\em J. Inst. Math. Jussieu 3}, 2 (2004), 231--314.

\bibitem{Morrow_pro_GL}
{\sc Morrow, M.}
\newblock The ${K}$-theory and logarithmic {H}odge--{W}itt sheaves of formal
  schemes in characteristic $p$.
\newblock {\em {\tt 1512.04703}\/} (2015).

\bibitem{Morrow_BMSnotes}
{\sc Morrow, M.}
\newblock Notes on the {$\mathbb{A}_\sub{inf}$}-cohomology of ``{I}ntegral
  $p$-adic {H}odge theory''.
\newblock {\em {\tt arXiv:1608.00922}\/} (2016).

\bibitem{Scholze2012}
{\sc Scholze, P.}
\newblock Perfectoid spaces.
\newblock {\em Publ. Math. Inst. Hautes \'Etudes Sci. 116\/} (2012), 245--313.

\bibitem{Scholze2013}
{\sc Scholze, P.}
\newblock {$p$}-adic {H}odge theory for rigid-analytic varieties.
\newblock {\em Forum Math. Pi 1\/} (2013), e1, 77.

\bibitem{vanderKallen1986}
{\sc van~der Kallen, W.}
\newblock Descent for the {$K$}-theory of polynomial rings.
\newblock {\em Math. Z. 191}, 3 (1986), 405--415.

\bibitem{CesnaviciusKoshikawa2017}
{\sc \v{C}esnavi\v{c}ius, K., and Koshikawa, T.}
\newblock The {$\mathbb{A}_\sub{inf}$}-cohomology in the semistable case.
\newblock {\em {\tt arXiv:1710.06145}\/} (2017).

\end{thebibliography}

\noindent Matthew Morrow\\
Institut de Mathématiques de Jussieu--Paris Rive Gauche\\
UPMC - 4 place Jussieu\\
Case 247\\
75252 Paris\\
\\
{\tt matthew.morrow@imj-prg.fr}

\end{document}